\documentclass[10pt]{article}

\usepackage{xcolor}
\usepackage{fullpage,amsmath,amsfonts,graphicx,amssymb,amsthm}
\usepackage{enumitem}
\usepackage{algorithm}  
\usepackage{algorithmic}  
\usepackage{xr-hyper,hyperref}
\usepackage{graphicx,subfig}
\usepackage{authblk}
\usepackage{mathrsfs,mathtools,sansmath,nicefrac,sansmath}
\usepackage{wrapfig}

\allowdisplaybreaks
\newcommand\numberthis{\addtocounter{equation}{1}\tag{\theequation}}

\newtheorem{theorem}{Theorem}[section]

\newtheorem{lemma}[theorem]{Lemma}

\newtheorem{corollary}[theorem]{Corollary}
\newtheorem{definition}[theorem]{Definition}
\newtheorem{remark}[theorem]{Remark}

\def\la{\langle}
\def\ra{\rangle}
\def\lb{\left(}
\def\rb{\right)}
\def\lcb{\left\{}
\def\rcb{\right\}}

\def\lsb{\left[}
\def\rsb{\right]}

   \def\mR{\bfm R}  \def\R{\mathbb{R}}

\def \tran {\mathsf{T}}

\def\mR{\mathbb{R}}
\def\st{\text{\quad subject to\quad}}

\def\bfC{\mathbf{C}}
\def\bfP{\mathbf{P}}
\def\bfA{\mathbf{A}}
\def\bfB{\mathbf{B}}
\def\bff{\mathbf{f}}
\def\bfg{\mathbf{g}}
\def\bfx{\mathbf{x}}
\def\bfy{\mathbf{y}}
\def\ones{\boldsymbol{1}}
\def\zeros{\boldsymbol{0}}
\def\bfK{\mathbf{K}}
\def\bfA{\mathbf{A}}
\def\bfu{\boldsymbol{u}}
\def\bfv{\boldsymbol{v}}
\def\bfa{\boldsymbol{a}}
\def\bfb{\boldsymbol{b}}

\def\bfx{\boldsymbol{x}}
\def\bfy{\boldsymbol{y}}
\def\bfX{\mathbf{X}}
\def\bfY{\mathbf{Y}}
\def\bfalpha{\boldsymbol{\alpha}}
\def\bfbeta{\boldsymbol{\beta}}
\def\bfPi{\boldsymbol{\Pi}}

\def\mR{\mathbb{R}}

\def\st{\text{\quad subject to\quad}}

\DeclareMathOperator{\diag}{\mathrm{diag}}

\begin{document}
\title{Improved Complexity Analysis of the Sinkhorn and Greenkhorn Algorithms for Optimal Transport\footnotetext{JL and DY contribute equally.}}

\author[1]{Jianzhou Luo}
\author[1]{Dingchuan Yang}
\author[1]{Ke Wei}
\affil[1]{School of Data Science, Fudan University, Shanghai, China.}

\date{\today}

\maketitle

\begin{abstract}
\noindent
The Sinkhorn algorithm is a widely used method for solving the optimal transport problem, and the Greenkhorn algorithm is one of its variants. While there are modified versions of these two algorithms whose computational complexities are $O({n^2\|\bfC\|_\infty^2\log n}/{\varepsilon^2})$ to achieve an $\varepsilon$-accuracy, to the best of our knowledge, the existing complexities for the vanilla versions are $O({n^2\|\bfC\|_\infty^3\log n}/{\varepsilon^3})$. In this paper we fill this gap and show that the complexities of the vanilla Sinkhorn and Greenkhorn algorithms are indeed $O({n^2\|\bfC\|_\infty^2\log n}/{\varepsilon^2})$. The analysis relies on the equicontinuity of the dual variables for the discrete entropic regularized optimal transport problem, which is of independent interest.\\

\noindent
\textbf{Keywords.} Optimal transport, Sinkhorn algorithm, Greenkhorn algorithm, complexity analysis
\end{abstract}



\section{Introduction}
Optimal transport (OT) is the problem of finding the most economic way to transfer one measure to another. The optimal transport problem dates back to Gaspard Monge \cite{monge1781memoire}.  Kantorovich (see \cite{peyre2019computational} for historical context)  provides a relaxation of this problem  by allowing the mass splitting in  source based on the
notation of coupling, which advances its applications in  areas such as economics and physics. Recently, with the increasing demand of quantifying the dissimilarity for probability measures, the optimal transport cost, which can serve as a natural notion of metric between measures, has attracted a lot of attention in the fields of statistics and machine learning. For instance, as a special case of the optimal transport cost, Wasserstein distance \cite{villani2008optimal} can be used as a loss function when the output of a supervised learning method is a probability measure \cite{frogner2015learning}. The application of Wasserstein-1 metric in GAN can effectively mitigate the issues of vanishing gradients and mode collapse \cite{arjovsky2017wasserstein}. Other applications of optimal transport include shape matching \cite{su2015optimal}, color transfer \cite{solomon2015convolutional} and object detection \cite{ge2021ota}.

For ease of notation, we consider the case where the target probability vectors are of the same length, denoted $\bfa,\bfb \in \Delta_n$. Given a cost matrix $\bfC\in\R_+^{n\times n}$, the Kantorovich relaxation attempts to compute the distance between $\bfa$ and $\bfb$ by solving the following problem:
\begin{align}\label{Kan}
    \min_{\bfP\in\bfPi(\bfa,\bfb)}\la\bfC,\bfP\ra,
\end{align}
where
\begin{align*}
    \bfPi(\bfa,\bfb):=\lcb\bfP\in\mR_+^{n\times n}:\bfP\ones_n = \bfa,\bfP^\tran\ones_n = \bfb\rcb.
\end{align*}

As a linear programming, the Kantorovich problem can be solved by the simplex method or the interior-point method. Yet, both of them suffer from huge computational complexity (overall $O(n^3)$) \cite{peyre2019computational,pele2009fast}. A computationally efficient algorithm is developed  in \cite{cuturi13} which  solves the Shannon entropic regularized Kantorovich problem
\begin{align}\label{entroKan}
    \min_{\bfP\in\bfPi(\bfa,\bfb)}\la\bfC,\bfP\ra - \gamma H(\bfP)
\end{align}
by the Sinkhorn update, where $\gamma$ and $H(\bfP) = -\sum_{i,j}\bfP_{i,j}(\log \bfP_{i,j}-1)$ are the \textit{regularization parameter} and \textit{discrete entropy}, respectively. The easily parallel and GPU friendly features of the Sinkhorn algorithm make it popular in applications, especially in machine learning \cite{frogner2015learning}. The idea of the Sinkhorn update  can date back to \cite{yule1912methods,sinkhorn1964relationship}, and there are many studies on its  convergence. The linear convergence rate of it (not for optimal transport) under Hilbert projective metric is provided in \cite{franklin1989scaling}. Regarding optimal transport, the computational complexity of the Sinkhorn algorithm to achive an $\varepsilon$-accuracy is shown to be $O({n^2\|\bfC\|_\infty^3\log n}/{\varepsilon^3})$ \cite{altschuler2017near}. The Greenkhorn algorithm is an  elementwise variant of the Sinkhorn algorithm which has the same computational complexity \cite{altschuler2017near}. In \cite{pmlr-v80-dvurechensky18a}, a variant of the Sinkhorn algorithm which utilizes  modified probability vectors is developed and the $O({n^2\|\bfC\|_\infty^2\log n}/{\varepsilon^2})$ computational complexity is established. The idea has been extended to the Greenkhorn algorithm with the same improved complexity \cite{lin2022efficiency}. However, as can be seen later, the numerical results show that the practical performances of the original methods and the modified ones are overall indistinguishable, which motivates us to think whether we can improve the computational complexity of the vanilla Sinkhorn and Greenkhorn algorithms. The answer is affirmative. 

\textbf{Main contributions.} 
 Inspired by the equicontinuity for the continuous KL-divergence regularized OT problem, in this paper we have identified a discrete analogue of this property and shown that the $O({n^2\|\bfC\|_\infty^2\log n}/{\varepsilon^2})$ complexity of the vanilla Sinkhorn algorithm can be achieved based on this property.
    In addition, it is shown that this property can also be used to improve the complexity of the vanilla Greenkhorn algorithm to
    $O({n^2\|\bfC\|_\infty^2\log n}/{\varepsilon^2})$. 

\textbf{Notation.}
Let $\R_+$ and $\R_{++}$ denote non-negative real numbers and positive real numbers respectively. For a vector $\bfx\in\R^n$, we denote by $\|\bfx\|_1 :=\sum_{i=1}^n|\bfx_i|$ the sum of absolute values its elements, and by $\|\bfx \|_\infty$ the maximum absolute value of its elements. Given a matrix $\bfA \in \R^{n\times n}$, we write $\|\bfA\|_1:=\sum_{ij}|\bfA_{ij}|$ and $\| \bfA \|_{\infty} := \max_{ij}|\bfA|_{ij}$. Let $\Delta_n:=\{\bfx\in\mR^n_+:\|\bfx\|_1=1\}$ be the probability simplex in $\mR^n$. For a matrix $\bfA$ and a vector $\bfx$, we denote by $e^{\bfA}$, $ e^{\bfx}$, $\log\bfA$, $\log\bfx$ their entrywise exponentials and natural logarithms respectively.  For vectors $\bfx,\bfy$ of the same dimension, their product $\bfx \odot \bfy$ and division $\frac{\bfx}{\bfy}$ should be interpreted in an entrywise way. For two probability vectors $\bfa,\bfb \in \Delta_n$,  KL-divergence is defined as
\[
KL(\bfa\|\bfb) := \sum_{i=1}^n \bfa_i \log\frac{\bfa_i}{\bfb_i}.
\]
The Pinsker's inequality provides a lower bound of $KL(\bfa\|\bfb)$ in terms of the $\ell_1$-norm:
\[
    KL(\bfa\|\bfb)\geq\frac{1}{2}\|\bfa-\bfb\|^2_1.
\]

\section{Sinkhorn  and Its Complexity}\label{sec:sink}
\begin{algorithm}[ht!]
    \renewcommand{\algorithmicrequire}{\textbf{Input:}}
    \renewcommand{\algorithmicensure}{\textbf{Output:}}
    \caption{Sinkhorn \cite{cuturi13}}
    \label{Alg:sinkhorn}
    \begin{algorithmic}
        \REQUIRE probability vectors $\bfa,\bfb$, accuracy $\delta$, $\bfK = \exp(-\frac{\bfC}{\gamma})$.
        \STATE Initialize $k\leftarrow 0$, $\bfu_0 > 0$, $\bfv_0 > 0$ 
        \REPEAT 
        \IF{$k$ mod $2 = 0$} 
        \STATE $\bfu_{k+1} = \frac{\bfa}{\bfK \bfv_k}$
        \STATE $\bfv_{k+1} = \bfv_k$
        \ELSE
        \STATE $\bfu_{k+1} = \bfu_k$
        \STATE $\bfv_{k+1} = \frac{\bfb}{\bfK^\tran \bfu_k}$
        \ENDIF
        \STATE $k \leftarrow k+1$
            \STATE $\bfP_k \leftarrow \diag(\bfu_k) \bfK \diag(\bfv_k)$
        \UNTIL {$\| \bfP_k \ones_n - \bfa \|_1 + \|\bfP_k^\tran \ones_n - \bfb \|_1 \leq \delta$}
        \ENSURE $\bfP_k$.
    \end{algorithmic}
\end{algorithm}
Without loss of generality, assume $\bfa >0$ and $\bfb >0$\footnote{
The case where there are zeros in $\bfa$ or $\bfb$ is discussed in Remark \ref{remarkSink}.
}. Then the solution to the entropic regularized Kantorovich problem \eqref{entroKan} is the unique matrix of the form
\begin{equation}\label{solution:uvform}
    \bfP = \diag(\bfu)\bfK\diag(\bfv)
\end{equation}
that satisfies 
\begin{align*}
\bfu \odot(\bfK\bfv) = \bfa \quad \text{and} \quad \bfv \odot (\bfK^\tran \bfu) = \bfb.\numberthis\label{uvconstraints}
\end{align*}
Here, $\bfK=\exp(-\frac{\bfC}{\gamma})$, $(\bfu,\bfv)\in\R_{++}^n \times \R_{++}^n$ are two (unknown) scaling vectors, and \eqref{uvconstraints} is a reformulation of the constraints $\bfP\ones_n = \bfa,\bfP^\tran\ones_n = \bfb$ in terms of $\bfu$ and $\bfv$. For the details of derivation, see \cite{peyre2019computational}. 
Based on this optimality condition, the Sinkhorn algorithm  with arbitrary positive initialization vectors $\bfu_0$ and $\bfv_0$ conducts alternating projection to satisfy  \eqref{uvconstraints}, see Algorithm \ref{Alg:sinkhorn} for a description of the algorithm. As already mentioned, the best known complexity for the vanilla Sinkhorn algorithm is $O({n^2\|\bfC\|_\infty^3\log n}/{\varepsilon^3})$ \cite{altschuler2017near}, while this complexity can be improved to $O({n^2\|\bfC\|_\infty^2\log n}/{\varepsilon^2})$ for the modified variant which uses lifted values of $\bfa$ and $\bfb$ \cite{pmlr-v80-dvurechensky18a} (more precisely, run Sinkhorn using $(1 - \frac{\delta}{8})((\bfa,\bfb) + \frac{\delta}{n(8-\delta)}(\ones,\ones))$ for a small mismatch $\delta$ instead of ($\bfa$, $\bfb$)). However, the numerical results in Figure~\ref{fig:sinkhorn}  demonstrate that the practical performances of the vanilla and modified Sinkhorn algorithms are overall similar, which motivates us to improve the convergence analysis of the vanilla Sinkhorn algorithm.
{The numerical experiments are conducted on two data sets: the MNIST data set and a widely used synthetic data set\cite{altschuler2017near}. In the experiments, a pair of images are randomly sampled from each data set and $\bfa$ and $\bfb$ are obtained by vectorizing the images, followed by normalization. The transport cost is computed as the $\ell_2$-distance of the pixel positions. The plots in Figure~\ref{fig:sinkhorn} are average results over 10 random instances.}


\begin{figure}[ht!]
    \centering
    \subfloat[MNIST data set]{\includegraphics[width = 0.5\textwidth]{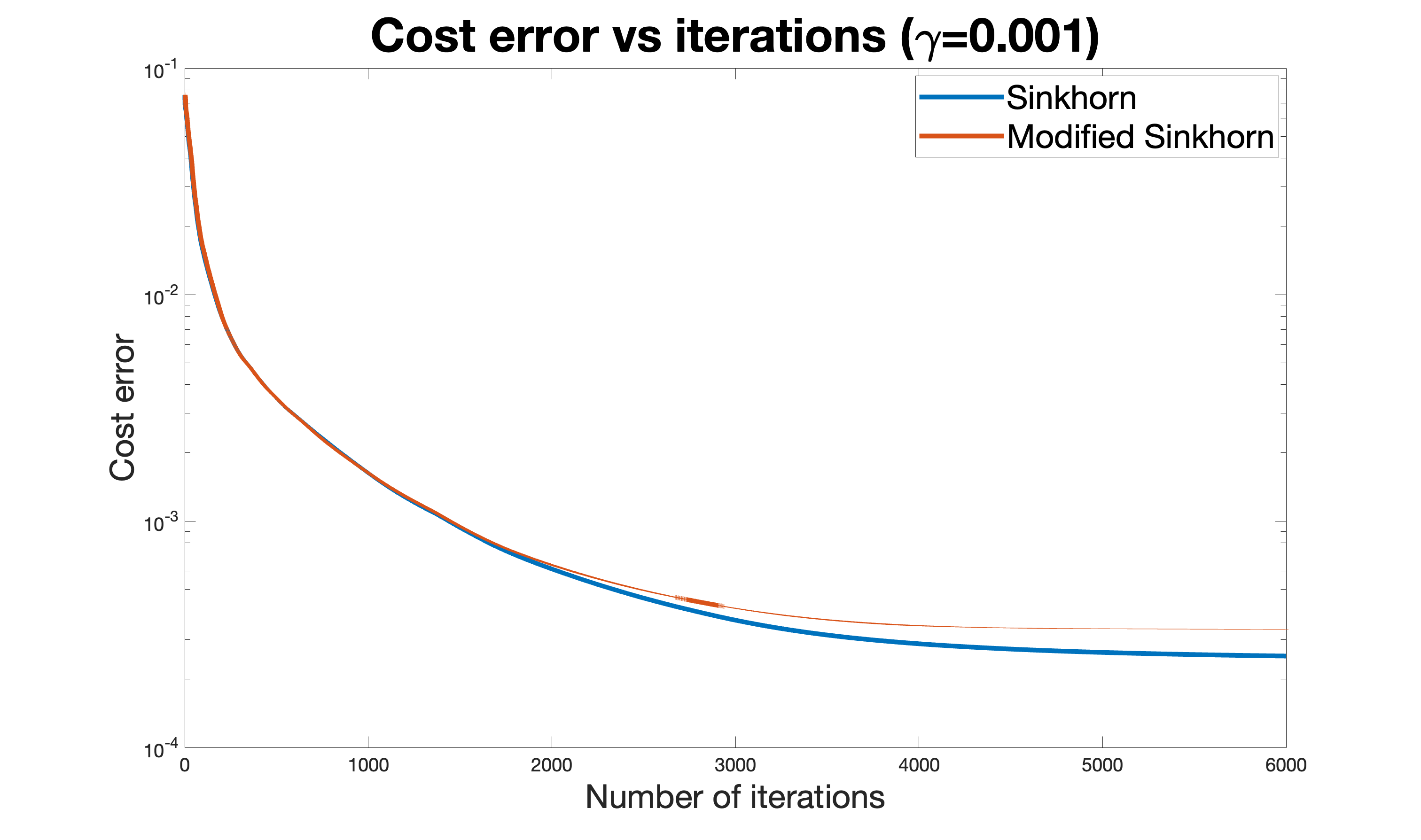}}
    \subfloat[Synthetic data set]{\includegraphics[width = 0.5\textwidth]{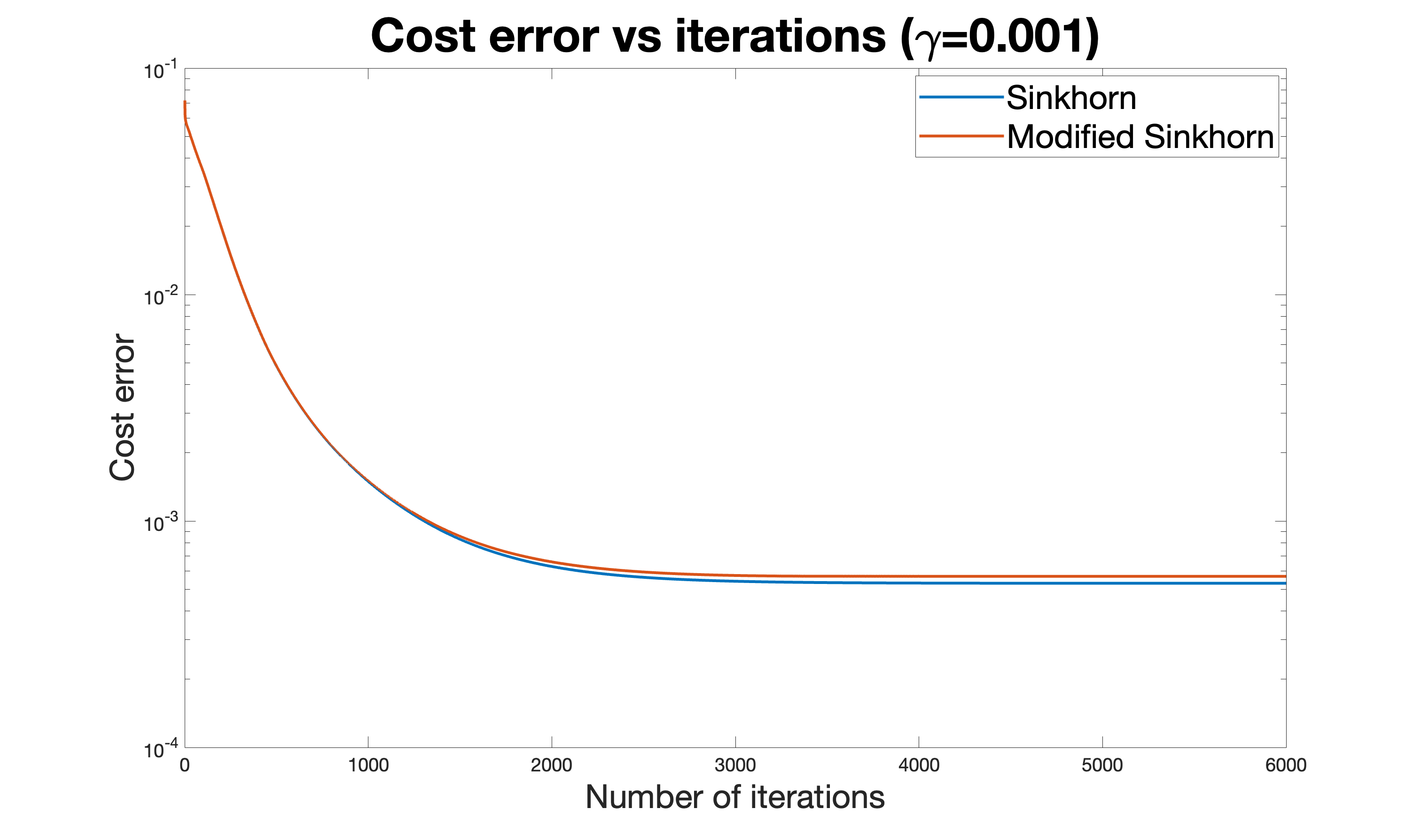}}
    \caption{Transport cost error {\it vs} number of iteration for the vanilla and modified Sinkhorn algorithms.}
    \label{fig:sinkhorn}
\end{figure}

There are several different equivalent interpretations of the Sinkhorn algorithm \cite{peyre2019computational}, and the convergence analysis is based on the dual perspective. First note that the dual problem of the entropic regularized Kantorovich problem \eqref{entroKan} is given by
\begin{align}\label{dualentro}
    \max_{\bff,\bfg\in\mR^n} h(\bff,\bfg):=\la\bff,\bfa\ra+\la\bfg,\bfb\ra - \gamma\sum_{i,j}\exp\lsb\frac{1}{\gamma}(\bff_i+\bfg_j-\bfC_{ij})\rsb.
\end{align}
For the details of derivation, see for example \cite{cuturi13}.
If we define $$\bff_k = \gamma\log\bfu_k, \;\bfg_k = \gamma\log\bfv_k,$$ 
it can be easily verified that
\begin{equation*}
\begin{cases}
    \bff_{k+1} = \gamma\log\bfa - \gamma\log (\bfK e^{\frac{\bfg_k}{\gamma}}), \quad\bfg_{k+1}=\bfg_k, \quad \text{if $k$ is even},\\
    \bff_{k+1} = \bff_k, \quad \bfg_{k+1} = \gamma\log\bfb - \gamma\log (\bfK^\tran e^{\frac{\bff_k}{\gamma}}), \quad \text{if $k$ is odd}.
\end{cases}
\end{equation*}
Moreover, when $k\geq 2$,
\[\begin{cases}
    \nabla_{\bff} h(\bff_{k+1},\bfg_{k+1}) = \bfa - e^{\bff_{k+1}/\gamma}\odot (\bfK e^{\bfg_{k+1}/\gamma}) = \zeros,\quad \text{if $k$ is even},\\
    \nabla_{\bfg} h(\bff_{k+1},\bfg_{k+1}) = \bfb - e^{\bfg_{k+1}/\gamma}\odot (\bfK^\tran e^{\bff_{k+1}/\gamma}) = \zeros,\quad \text{if $k$ is odd}.
\end{cases}
\]
Therefore, in terms of $(\bff_k,\bfg_k)$, the Sinkhorn algorithm can be interpreted as a block coordinate ascent on the dual problem \eqref{dualentro}, and the analysis will be primarily based on $(\bff_k,\bfg_k)$. 

\begin{algorithm}[ht!]
    \renewcommand{\algorithmicrequire}{\textbf{Input:}}
    \renewcommand{\algorithmicensure}{\textbf{Output:}}
    \caption{$\mbox{Round}(\bfP,\bfa,\bfb)$ \cite{altschuler2017near}}  
    \label{alg:A}  
    \begin{algorithmic}
        \REQUIRE $\bfP\in\mR_+^{n\times n}$, $\bfa,\bfb\in\mR^n_+$.
        \FOR{$i\in [n]$}
        \STATE\textbf{if} $(\bfP\ones_n)_i>\bfa_i$ \textbf{then} $\bfx_i = \frac{\bfa_i}{(\bfP\ones)_i}$; \textbf{else} $\bfx_i=1$
        \ENDFOR
        \STATE $\bfX \leftarrow \diag(\bfx)$, $\bfP'\leftarrow\bfX\bfP$
        \FOR{$j \in [n]$}
        \STATE \textbf{if} {$(\bfP^{\tran}\ones_n)_j > \bfb_j$} \textbf{then} $\bfy_j= \frac{\bfb_j}{(\bfP^{\tran}\ones_n)_j}$; \textbf{else} $\bfy_j = 1$
        \ENDFOR
        \STATE $\bfY \leftarrow \diag(\bfy)$, $\bfP''\leftarrow\bfP'\bfY$
        \STATE {$\Delta_{\bfa}\leftarrow\bfa-\bfP''\ones_n,\Delta_{\bfb}\leftarrow\bfb-\bfP''^\tran\ones_n$}
        \STATE {$\widehat{\bfP}\leftarrow\bfP''+\frac{1}{\|\Delta_{\bfa}\|_1}\Delta_{\bfa}\Delta_{\bfb}^\tran$}
        \ENSURE $\widehat{\bfP}\in\bfPi(\bfa,\bfb)$.
    \end{algorithmic}
\end{algorithm}

After computing an approximate transport plan matrix $\bfP_k$ which satisfies
 $\|\bfP_k \ones_n - \bfa \|_1 + \|\bfP_k^\tran \ones_n - \bfb \|_1 \leq \delta$ by the Sinkhorn algorithm, a rounding step presented in Algorithm~\ref{alg:A} usually follows to ensure that the output matrix $\widehat{\bfP}_k$ satisfies $\widehat{\bfP}_k\ones_n=\bfa$ and $\widehat{\bfP}_k^\tran\ones_n=\bfb$ simultaneously. 
The goal in the analysis is to study the computational complexity of the algorithm when  an $\varepsilon$-approximate solution of the original Kantorovich problem \eqref{Kan} is obtained, that is when $\widehat{\bfP}_k$ satisfies
\[
\langle \bfC,\widehat{\bfP}_k \rangle \leq \langle \bfC,\bfP^* \rangle + \varepsilon,
\]
where $\bfP^*$ is the solution to \eqref{Kan}.

To carry out the analysis, we first list three lemmas from \cite{altschuler2017near}. For readers' convenience, we have included the proofs in the appendix. The first
lemma concerns the error bound of the rounding scheme, which is presented in a general form where $\bfa$ and $\bfb$ are not necessarily probability vectors. This will be useful in the analysis of the Greenkhorn algorithm in the next section.
\begin{lemma}[\protect{\cite[Lemma 7]{altschuler2017near}} ]\label{round_error}
For any vectors $\bfa,\bfb\in\R_{+}^n$ satisfying $\|\bfa\|_1=\|\bfb\|_1$ and matrix $\bfP\in\mR^{n\times n}_{+}$ , Algorithm \ref{alg:A} outputs a matrix $\rm{Round}(\bfP,\bfa,\bfb)\in\bfPi(\bfa,\bfb)$ such that
\begin{align*}
    \|\bfP-\rm{Round}(\bfP,\bfa,\bfb)\|_1\leq 2\lb\|\bfa-\bfP\ones_n\|_1+\|\bfb - \bfP^\tran\ones_n\|_1\rb.
\end{align*}
\end{lemma}
 The second lemma provides a bound for the error between the true OT distance $\langle \bfC, \bfP^*\rangle$ and the approximate OT distance $\langle \bfC,\widehat{\bfP}_k \rangle$.
\begin{lemma}[\protect{\cite[Theorem 1]{altschuler2017near}} ]\label{thm:sinkbound}
    Let $\bfP^*$ be the solution to the Kantorovich problem \eqref{Kan} with marginals $\bfa$, $\bfb$ and cost matrix $\bfC$. Let $\bfB$ be a probability matrix \mbox{\normalfont (i.e. $\|\bfB\|_1=1$)} of the form $\bfB = \rm{diag}(\bfu)\bfK\rm{diag}(\bfv)$, where $\bfu,\bfv \in \R^n_{++}$ and $\bfK \in \R^{n\times n}$ is defined by $\bfK = \exp(-\frac{\bfC}{\gamma})$. Letting $\widehat{\bfB}=\mbox{\normalfont Round}(\bfB,\bfa,\bfb)$, we have
    \begin{align}\label{sinkbound}
        \langle \bfC, \widehat{\bfB}\rangle \leq \langle \bfC,\bfP^* \rangle + 2\gamma\log n + 4(\| \bfB \ones_n - \bfa \|_1 + \|\bfB^\tran \ones_n - \bfb \|_1)\|\bfC\|_\infty.
    \end{align}
\end{lemma}
\begin{remark}\label{remark:tmp_ke1}
In \eqref{sinkbound}, If we set $\bfB$ to be $\bfP_k$, the output of Algorithm~\ref{Alg:sinkhorn}, we have 
\[
\langle \bfC, \widehat{\bfP}_k\rangle \leq \langle \bfC,\bfP^* \rangle + 2\gamma\log n +4(\| \bfP_k \ones_n - \bfa \|_1 + \|\bfP_k^\tran \ones_n - \bfb \|_1)\|\bfC\|_\infty.
\]
This implies that if we set $\gamma = \frac{\varepsilon}{4\log n}$ for the regularization parameter and choose  $\delta = \frac{\varepsilon}{8\|\bfC\|_\infty}$ as the termination condition, the Sinkhorn algorithm will output a $\bfP_k$ such that $$\langle\bfC,\mbox{\normalfont Round}(\bfP_k,\bfa,\bfb)\rangle\leq\langle \bfC,\bfP^* \rangle+\varepsilon.$$ 
\end{remark}
The third lemma characterizes the one-iteration improvement of the Sinkhorn algorithm in  terms of the function value. 
\begin{lemma}[\protect{\cite[Lemma 2]{altschuler2017near}} ]\label{deltaimprove}
    If $k\geq 2$, then
    \begin{align*}
         h(\bff_{k+1},\bfg_{k+1}) - h(\bff_k,\bfg_k) \geq \frac{\gamma}{2}(\| \bfP_k \ones_n - \bfa \|_1 + \|\bfP_k^\tran \ones_n - \bfb \|_1)^2.
    \end{align*}
\end{lemma}

The improved convergence analysis of the Sinkhorn algorithm relies heavily on the equicontinuity of the discrete regularized optimal transport problem. To state this property, we first follow the convention in the optimal transport literature\cite{santambrogio2015optimal} and define the $(\bfC,\gamma)$-transform and $\overline{(\bfC,\gamma)}$-transform of vectors.
\begin{definition}[$(\bfC,\gamma)$-transform, $\overline{(\bfC,\gamma)}$-transform]\label{Cgamma}
For any vector $\bff\in\mR^n$, we define its $(\bfC,\gamma)$-transform by
\begin{align*}
    \bff^{(\bfC,\gamma)}:=\mathop{\arg\max}_{\bfg\in\mR^n}h(\bff,\bfg).
\end{align*}
For any vector $\bfg\in\mR^n$, we define its $\overline{(\bfC,\gamma)}$-transform  by
\begin{align*}
    \bfg^{\overline{(\bfC,\gamma)}}:=\mathop{\arg\max}_{\bff\in\mR^n}h(\bff,\bfg).
\end{align*}
\end{definition}
The definitions of $(\bfC,\gamma)$-transform and $\overline{(\bfC,\gamma)}$-transform enable us to describe the Sinkhorn update as
\begin{align*}
        \begin{split}
            \left \{
            \begin{array}{cc}
                \bff_{k+1} = \bfg_k^{\overline{(\bfC,\gamma)}},\textbf{ }\bfg_{k+1} = \bfg_k, & \text{if } k \text{ is even},\\
                \\
                \bff_{k+1} = \bff_k,\textbf{ }\bfg_{k+1} = \bff_k^{(\bfC,\gamma)}, & \text{if }k \text{ is odd}.
            \end{array}
            \right.
        \end{split}
    \end{align*} 
     Furthermore, we have
    \begin{align*}
        \bff_*^\gamma = \lb\bfg_*^\gamma\rb^{\overline{(\bfC,\gamma)}},\quad \bfg_*^\gamma = \lb\bff_*^\gamma\rb^{ (\bfC,\gamma)},
    \end{align*}
    where $(\bff_*^\gamma,\bfg_*^\gamma)$ is any solution of \eqref{dualentro}. Next we present the key property used in the analysis, which is inspired  by the equicontinuity of the $(c,\gamma)$-transform and $\overline{(c,\gamma)}$-transform for the continuous  KL-divergence regularized optimal transport problem. 
\begin{lemma}[Equicontinuity]\label{discret_equic} Assume $h$ in \eqref{dualentro} is defined with $\bfa>0$, $\bfb>0$, and $\bfC\geq 0$.
    Then for any vectors $\bff,\bfg\in\mR^n$, 
    \begin{align*}
        &\max(\bff^{(\bfC,\gamma)}-\gamma\log \bfb)-\min(\bff^{(\bfC,\gamma)}-\gamma\log \bfb)\leq \max(\bfC)-\min(\bfC)\leq \|\bfC\|_\infty,\\
        &\max(\bfg^{\overline{(\bfC,\gamma)}}-\gamma\log \bfa)-\min(\bfg^{\overline{(\bfC,\gamma)}}-\gamma\log \bfa)\leq\max(\bfC)-\min(\bfC)\leq \|\bfC\|_\infty,
    \end{align*}
    where $\max(\cdot)$ and $\min(\cdot)$ refer to the largest entry and smallest entry of a vector or matrix, respectively.
\end{lemma}
\begin{proof}
    By the definition of $(\bfC,\gamma)$-transform, we have
    \begin{align*}
        &\lsb\nabla_g h(\bff,\bff^{(\bfC,\gamma)})\rsb_j = \bfb_j - \sum_{i=1}^n\exp\lsb\frac{1}{\gamma}(\bff_i+\bff^{(\bfC,\gamma)}_j-\bfC_{i,j})\rsb=0
    \end{align*}
    for any $j\in [n]$. This means for any $j,j'\in [n]$,
    \begin{align}\label{bothone}
        &\sum_{i=1}^n\exp\lsb\frac{1}{\gamma}(\bff_i+\bff^{(\bfC,\gamma)}_j-\gamma\log\bfb_j-\bfC_{i,j})\rsb=\sum_{i=1}^n\exp\lsb\frac{1}{\gamma}(\bff_i+\bff^{(\bfC,\gamma)}_{j'}-\gamma\log\bfb_{j'}-\bfC_{i,{j'}})\rsb=1.
    \end{align}
     Suppose that there exist $j,j'\in [n]$ satisfying 
    \begin{align*}
        \bff^{(\bfC,\gamma)}_j-\gamma\log\bfb_j > \bff^{(\bfC,\gamma)}_{j'}-\gamma\log\bfb_{j'}+\max(\bfC)-\min(\bfC).
    \end{align*}
    Then we have
    \begin{align*}
        \exp\lsb\frac{1}{\gamma}(\bff_i+\bff^{(\bfC,\gamma)}_j-\gamma\log\bfb_j-\bfC_{i,j})\rsb > \exp\lsb\frac{1}{\gamma}(\bff_i+\bff^{(\bfC,\gamma)}_{j'}-\gamma\log\bfb_{j'}-\bfC_{i,{j'}})\rsb
    \end{align*}
    for any $i\in[n]$, which contradicts \eqref{bothone} after summing it up for all $i\in[n]$. This completes the proof of the first claim, and the second claim can be proved similarly.
\end{proof}
\begin{remark}
    The upper bound in Lemma \ref{discret_equic} is tight. Here we give an example. Let $\bff = \zeros$. By the definition of the $(\bfC,\gamma)$-transform, we have \begin{align*}
        \bff^{(\bfC,\gamma)}_j-\gamma\log\bfb_j=-\gamma\log\sum_i\exp\lb-\frac{1}{\gamma}\bfC_{i,j}\rb
    \end{align*}
    for any $j\in[n]$. Consider  the case where the entries of the $j$-th column of $\bfC$ are all equal to $\|\bfC\|_\infty$ and the entries of the $j'$-th column of $\bfC$ are all zeros. Then a simple calculation yields 
    \begin{align*}
        \bff^{(\bfC,\gamma)}_j-\gamma\log\bfb_j=\bff^{(\bfC,\gamma)}_{j'}-\gamma\log\bfb_{j'}+\|\bfC\|_\infty.
    \end{align*}
\end{remark}
\begin{remark}
    Let $a$, $b$ be two regular Borel measures supported on a compact metric spaces $(X,d_X)$ and $(Y,d_Y)$, respectively, and $c\in C(X\times Y)$ be a continuous function defined on $X\times Y$. The continuous KL-divergence regularized Kantorovich problem is defined by
\begin{align}\label{C_KL}
    \min_{P\in\pi(a,b)}\int_{X\times Y}c(x,y) dP(x,y) + \gamma KL(P\| a\otimes b),
\end{align}
where the $KL$-divergence between measures is defined by
\[
    KL(\mu\|\nu):=\int_{X\times Y}\log\lb\frac{d\mu}{d\nu}(x,y)\rb d\mu(x,y)
\]
for $\mu,\nu\in\mathcal{P}(X\times Y)$. The dual of (\ref{C_KL}) can is given by (see for example \cite{peyre2019computational})
\begin{align*}
    \max_{\alpha\in C(x), \beta\in C(Y)}\xi(\alpha,\beta):=\int_X \alpha(x) da(x)+\int_Y \beta(y) db(y)-\gamma\int_{X\times Y}\exp\lb\frac{1}{\gamma}(\alpha(x)+ \beta(y) -c(x,y))\rb d (a\otimes b)(x,y),
\end{align*}
and the $(c,\gamma)$-transform and $\overline{(c,\gamma)}$-transform are partial maximizations of the functional $\xi(\alpha,\beta)$. We only consider the $(c,\gamma)$-transform here and the $\overline{(c,\gamma)}$-transform can be be similarly discussed. For any $\alpha:X\to \mR$,  its $(c,\gamma)$-transform $\alpha^{(c,\gamma)}:Y\to \mR$ is defined by
\begin{align*}
    \alpha^{(c,\gamma)}(y) := \mathop{\arg\max}_{s\in\mR} \lcb s - \gamma\int_X \exp\lb\frac{1}{\gamma}(\alpha(x)+s-c(x,y)\rb d a(x)\rcb.
\end{align*}
For any $\alpha:X\to \mR$ and $x\in X$, it has been shown that $\alpha^{(c,\gamma)}$  share the same modulus of continuity as $\{c(x,\cdot)\}_x$ \cite{marino2020optimal}. That is, if 
\begin{align*}
    |c(x,y)-c(x,y')|\leq \omega(d_Y(y,y'))
\end{align*}
for any $x\in X$ and $y,y'\in Y$, where $\omega:\mR_+\to\mR_+$ is an increasing continuous function with $\omega(0)=0$, then
\begin{align*}
    |\alpha^{(c,\gamma)}(y)-\alpha^{(c,\gamma)}(y')|\leq \omega(d_Y(y,y'))
\end{align*}
for any $y,y'\in Y$. 

When $X = \{x_1,\dots,x_n\},Y=\{y_1,\dots,y_n\}$, the measures $a$ and $b$ can be described by vectors
$\bfa,\bfb$ satisfying $\bfa_i=a(\{x_i\}),\bfb_j=b(\{y_j\})$ for any $i,j\in[n]$, and the function $c$ can be described by a matrix $\bfC\in \mR^{n\times n}$ satisfying $\bfC_{i,j}=c(x_i,y_j)$ for any $(i,j)\in [n]\times [n]$. It is noted in \cite{cuturi13}  that the KL-divergence regularized Kantorovich problem is equivalent to the entropic regularized Kantorovich problem in this situation. Let $\bfalpha = [\alpha(x_1),\dots\alpha(x_n)]^\tran$, $\bfbeta = [\beta(y_1),\dots\beta(y_n)]^\tran$. Then we have
\begin{align*}
    \xi(\alpha,\beta)&=\la \bfalpha,\bfa\ra + \la\bfbeta,\bfb\ra-\gamma\sum_{i,j}\exp\lb\frac{1}{\gamma}(\bfalpha_i+\bfbeta_j-\bfC_{i,j})\rb\bfa_i\bfb_j\\
    &=h(\bfalpha+\gamma\log\bfa,\bfbeta+\gamma\log\bfb) - \gamma\la\log\bfa,\bfa\ra - \gamma\la\log\bfb,\bfb\ra\\
    &= h(\bff,\bfg)- \gamma\la\log\bfa,\bfa\ra - \gamma\la\log\bfb,\bfb\ra,\quad \mbox{where }\bff=\bfalpha+\gamma\log\bfa\mbox{ and }\bfg=\bfbeta+\gamma\log\bfb.
\end{align*}
The equicontinuity property  with respect to $\bfalpha$  inherited from the continuous setting (under certain  $\omega$ and $d_Y$) inspires us to seek the equicontinuity property with respect to $\bff-\gamma\log \bfa$ (similar discussion  for $\bfg-\gamma\log\bfb$) directly for the discrete entropic regularized optimal transport problem,  as presented in Lemma~\ref{discret_equic}. Moreover, we will show that the complexities of the vanilla Sinkhorn and Greenkhorn algorithms can both be improved to
    $O({n^2\|\bfC\|_\infty^2\log n}/{\varepsilon^2})$ based on this property.
\end{remark}

Recall that in the Sinkhorn algorithm $\bff_k=\bfg_k^{\overline{(\bfC,\gamma)}}$ when $k\geq 2$ is odd and $\bfg_k=\bff_k^{(\bfC,\gamma)}$ when $k\geq 2$ is even. Moreover, 
since $\bff_k$ = $\bff_{k-1}$ if $k\geq 2$ is even and $\bfg_k$ = $\bfg_{k-1}$ if  $k\geq 2$ is odd, the following corollary is a straightforward consequence of Lemma \ref{discret_equic}. 
\begin{corollary}\label{maxf_c}
    No matter what initial variables $\bfu_0>0,\bfv_0>0$ are chosen in the Sinkhorn algorithm, for every $k\geq 2$, the corresponding dual variables satisfy
    \begin{align*}
        &\max(\bff_k-\gamma\log \bfa)-\min(\bff_k-\gamma\log \bfa)\leq \|C\|_\infty,\\
        &\max(\bfg_k-\gamma\log \bfb)-\min(\bfg_k-\gamma\log \bfb)\leq \|C\|_\infty.
    \end{align*}
    Moreover, for any solution $(\bff^\gamma_*,\bfg^\gamma_*)$ of problem \eqref{dualentro}, we have
    \begin{align*}
        &\max(\bff^\gamma_*-\gamma\log \bfa)-\min(\bff^\gamma_*-\gamma\log \bfa)\leq \|C\|_\infty,\\
        &\max(\bfg^\gamma_*-\gamma\log \bfb)-\min(\bfg^\gamma_*-\gamma\log \bfb)\leq \|C\|_\infty.
    \end{align*}
\end{corollary}
The following lemma provides a bound for the gap between $h(\bff_k,\bfg_k)$ and the optimal value $ h(\bff_*^\gamma,\bfg_*^\gamma)$ utilizing Corollary~\ref{maxf_c}.
\begin{lemma}\label{tkbound_c}
     Define $T_k := h(\bff_*^\gamma,\bfg_*^\gamma) - h(\bff_k,\bfg_k)$. Then for $k\geq 2$,
    \[
    T_k \leq \|\bfC\|_\infty(\|\bfa - \bfP_k\ones_n\|_1 + \|\bfb - \bfP_k^\tran\ones_n \|_1),
    \]
    where $\bfP_k=\diag(e^{\frac{\bff_k}{\gamma}})\bfK \diag(e^{\frac{\bfg_k}{\gamma}})$ is the corresponding transport plan matrix.
\end{lemma}
\begin{proof}
 First recall that 
\begin{align*}
  &\nabla_{\bff} h(\bff,\bfg) = \bfa - e^{\bff/\gamma}\odot (\bfK e^{\bfg/\gamma})=\bfa-\bfP\ones_n,\\
    &\nabla_{\bfg} h(\bff,\bfg) = \bfb - e^{\bfg/\gamma}\odot (\bfK^\tran e^{\bff/\gamma})=\bfb-\bfP^\tran\ones_n,
\end{align*}
where $\bfP=\diag(e^{\frac{\bff}{\gamma}})\bfK \diag(e^{\frac{\bfg}{\gamma}})$.
   Suppose without loss of generality that $k\geq 2$ is even. Since $h(\bff,\bfg)$ is a concave function,  we have
    \begin{align*}
        T_k &= h(\bff_*^\gamma,\bfg_*^\gamma) - h(\bff_k,\bfg_k)\\
        & \leq \langle\bff_*^\gamma - \bff_k, \bfa -\bfP_k\ones_n \rangle + \langle \bfg_*^\gamma - \bfg_k,  \bfb-\bfP_k^\tran\ones_n\rangle\\
        &= \langle\bff_*^\gamma - \bff_k, \bfa -\bfP_k\ones_n \rangle\\
        & = \langle(\bff_*^\gamma - \gamma\log\bfa ) - (\bff_k - \gamma \log \bfa), \bfa -\bfP_k\ones_n\rangle,\numberthis\label{eq:tmp_ke01}
    \end{align*}
    where the second equality follows from $\bfP_k^\tran\ones_n=\bfb$. Note that $\bfa$ and $\bfP_k\ones_n$ are both probability vectors. Thus,
    \begin{align*}
    \langle\bff_*^\gamma - \gamma\log\bfa , \bfa -\bfP_k\ones_n\rangle& = \langle \bff_*^\gamma - \gamma\log\bfa   - \lambda\ones_n, \bfa -\bfP_k\ones_n \rangle\\
    &\leq \|\bff_*^\gamma - \gamma\log\bfa   - \lambda\ones_n\|_\infty\|\bfa -\bfP_k\ones_n\|_1.
    \end{align*}
    Since it holds 
    for any $\lambda\in\R$, we have 
    \begin{align*}
    \langle\bff_*^\gamma - \gamma\log\bfa , \bfa -\bfP_k\ones_n\rangle&\leq \min_\lambda \|\bff_*^\gamma - \gamma\log\bfa   - \lambda\ones_n\|_\infty\|\bfa -\bfP_k\ones_n\|_1\\
    &=\frac{1}{2}\left[\max(\bff_*^\gamma - \gamma\log\bfa)-\min(\bff_*^\gamma - \gamma\log\bfa)\right]\|\bfa -\bfP_k\ones_n\|_1\\
    &\leq\frac{\|\bfC\|_\infty}{2}\|\bfa -\bfP_k\ones_n\|_1,
    \end{align*} where it is easy to see that the minimum is achieved at  $\lambda = \frac{1}{2}\lsb\max (\bff_*^\gamma - \gamma\log\bfa) + \min (\bff_*^\gamma - \gamma\log\bfa)\rsb$, and the last inequality follows from  Corollary \ref{maxf_c}. Since 
    $
\langle\bff_k - \gamma \log \bfa, \bfa -\bfP_k\ones_n\rangle
    $
    in \eqref{eq:tmp_ke01} can be bounded similarly, the proof is complete.
\end{proof}

\begin{remark}\label{lifting}
In \cite{pmlr-v80-dvurechensky18a}, it has been shown that 
\begin{align*}
    \max_i(\bff_k)_i -  \min_i(\bff_k)_i \leq R,\quad \max_i(\bfg_k)_j - \min_i(\bfg_k)_j \leq R,\\
    \max_i(\bff^*)_i -  \min_i(\bff^*)_i \leq R,\quad \max_i(\bfg^*)_j- \min_i(\bfg^*)_j \leq R
\end{align*}
and
\[
T_k \leq R(\|\bfa - \bfP_k\ones_n\|_1 + \|\bfb - \bfP_k^\tran\ones_n \|_1),
\]
where $R = -\log (\nu \min_{i,j}\lbrace \bfa_i,\bfb_j\rbrace)$ and $\nu = \min_{i,j}\bfK_{ij} = e^{-\frac{\|\bfC\|_\infty}{\gamma}}$. By modifying $\bfa, \bfb$ into $\tilde{\bfa},\tilde{\bfb}$, satisfying
\[
\| \tilde{\bfa} - \bfa \|_1 \leq \frac{\delta}{4},\quad \tilde{\bfb} - \bfb \|_1 \leq \frac{\delta}{4}
\]
and 
\[
\min_i\tilde{\bfa}_i \geq \frac{\delta}{8n},\quad \min_i\tilde{\bfb}_j \geq \frac{\delta}{8n},
\]
the Sinkhorn Algorithm using $\tilde{\bfa},\tilde{\bfb}$ achieves an $\varepsilon$-approximate solution to the original OT problem \ref{Kan} with complexity $O(\frac{n^2\|\bfC\|_\infty^2\log n}{\varepsilon^2})$. Such $\tilde{\bfa},\tilde{\bfb}$ exist. For example, we can choose
\[
(\tilde{\bfa},\tilde{\bfb}) = (1 - \frac{\delta}{8})((\bfa,\bfb) + \frac{\delta}{n(8-\delta)}(\ones,\ones)).
\]
However, using the equicontinuity property we have actually shown that $R$ can be replaced by $\|\bfC\|_\infty$, which is independent of $\bfa$ and $\bfb$. This fact enables us to prove that even for the vanilla Sinkhorn algorithm (without modifying $\bfa$ and $\bfb$),
the computational complexity  to achieve an $\varepsilon$-accuracy after rounding is $O(\frac{n^2\|\bfC\|_\infty^2\log n}{\varepsilon^2})$.
\end{remark}

\begin{theorem}\label{Sink_theorem}
      For any positive probability vectors $\bfa,\bfb \in \Delta_n$, Algorithm \ref{Alg:sinkhorn} outputs a matrix $\bfP_k$ such that $\|\bfa - \bfP_k\ones_n\|_1 + \|\bfb - \bfP_k^\tran\ones_n \|_1 \leq \delta$ in  less than
    $
    \lceil\frac{4\|\bfC\|_\infty}{\gamma\delta}\rceil+2
    $ number of iterations. 
    Furthermore, the computational complexity for the Sinkhorn algorithm to to achieve an $\varepsilon$-accuracy after rounding is $O\left( \frac{\|\bfC\|^2_\infty n^2 \log n}{\varepsilon^2}\right)$.
\end{theorem}
\begin{proof}
   It follows from Lemmas~\ref{deltaimprove} and \ref{tkbound_c} that
    \[
     T_k - T_{k+1} \geq \frac{\gamma}{2}(\| \bfP_k \ones_n - \bfa \|_1 + \|\bfP_k^\tran \ones_n - \bfb \|_1)^2
    \]
 and 
    \[
    T_k \leq \|\bfC\|_\infty(\|\bfa - \bfP_k\ones_n\|_1 + \|\bfb - \bfP_k^\tran\ones_n \|_1).
    \]
    Combining these two inequalities together yields
    \[
    T_k - T_{k+1} \geq \frac{\gamma}{2\|\bfC\|_\infty^2}T_k^2.
    \]
    Since $\{\frac{1}{T_k}\}_k$ is an increasing sequence, we have
    \[
    \frac{1}{T_{k+1}} - \frac{1}{T_k} = \frac{T_k - T_{k+1}}{T_kT_{k+1}} \geq\frac{T_k - T_{k+1}}{T_k^2}\geq \frac{\gamma}{2\|\bfC\|_{\infty}^2}.
    \]
    It follows immediately that
    \begin{align*}
        \frac{1}{T_k}  \geq \frac{1}{T_k}-\frac{1}{T_2} \geq \frac{(k-2)\gamma}{2\|\bfC\|_{\infty}^2}.
    \end{align*}
    Therefore, we have
    \begin{align*}
        T_k \leq \frac{2\|\bfC\|_{\infty}^2}{\gamma (k-2)}\quad\mbox{ for any $k\geq 3$.}
    \end{align*}
   Thus, if we choose $k_0 = 2 + \lceil\frac{2\|\bfC\|_{\infty}}{\gamma\delta}\rceil$, then 
    $
    T_{k_0}\leq \delta\|\bfC\|_\infty.
    $
    Applying Lemma \ref{deltaimprove} again shows that the iteration will terminate within at most
    $
    \lceil\delta\|\bfC\|_\infty / \frac{\gamma\delta^2}{2}\rceil = \lceil\frac{2\|\bfC\|_\infty}{\gamma\delta}\rceil
    $
    additional number of iterations. Therefore, the total number of iterations  required by the Sinkhorn algorithm  to terminate  is smaller than
    \[
     k_0 + \lceil\frac{2\|\bfC\|_\infty}{\gamma\delta}\rceil = 2 + 2\lceil\frac{2\|\bfC\|_\infty}{\gamma\delta}\rceil.
    \]
    Noting Remark~\ref{remark:tmp_ke1}, since the per iteration computational complexity of the Sinkhorn algorithm is $O(n^2)$, we can finally conclude that the total computational cost  for the Sinkhorn algorithm to achieve an $\varepsilon$-accuracy after rounding is $O\left( \frac{\|\bfC\|_\infty^2 n^2\log n}{\varepsilon^2}\right)$ by setting $\gamma = \frac{\varepsilon}{4\log n}$ and $\delta = \frac{\varepsilon}{8\|\bfC\|_\infty}$.
\end{proof}

\begin{remark}\label{remarkSink} 
 For the case when there are zeros in $\bfa$ and $\bfb$, define $A_1:=\{i\in[n]:\bfa_i=0\}$ and $A_2:=\{j\in[n]:\bfb_j=0\}$. Then for any $i\in A_1$, $j\in A_2$ and $k\geq 2$, it can be easily verified that $(\bfu_k)_i=0$ and $(\bfv_k)_j=0$. Thus the rows with index in $A_1$ and the columns with index in $A_2$ of each $\bfP_k$ and $\bfP^*$ are all zeros, and we only need to focus on the nonzero part of $\{\bfP_k\}_k$. It is equivalent to applying  the  Sinkhorn algorithm to solve the following compact problem
    \begin{align}\label{entroKantilde}
    \min_{\tilde{\bfP}\in\mR_+^{n_1\times n_2}}\la\tilde{\bfC},\tilde{\bfP}\ra + \gamma\sum_{i,j}\tilde{\bfP}_{i,j}(\log \tilde{\bfP}_{i,j}-1)\st \tilde{\bfP}\ones_{n_2}=\tilde{\bfa},\tilde{\bfP}^\tran\ones_{n_1}=\tilde{\bfb},
    \end{align}
    where $n_1=n-|A_1|,n_2 = n-|A_2|$, $\tilde{\bfa}\in \R_{++}^{n_1}$ and $\tilde{\bfb}\in \R_{++}^{n_2}$ are vectors obtained by removing the zeros from $\bfa$ and $\bfb$ respectively, and $\tilde{\bfC}\in \R_{+}^{n_1 \times n_2}$ is a matrix obtained by removing the rows with index in $A_1$ and the columns with index in $A_2$ from $\bfC$. Therefore the complexity in this case  is also $ O\left( \frac{\|\bfC\|^2_\infty n^2 \log n}{\varepsilon^2}\right)$.
\end{remark}

\begin{remark}
A set of experiments have been conducted to investigate the dependency of the computational complexity of Sinkhorn on the accuracy $\varepsilon$. See Figure~\ref{fig:sinkhorn_eps} for the plots of the average  number of iterations out of $10$ random trials needed for Sinkhorn to achieve an $\varepsilon$-accuracy against $1/\varepsilon^2$.  A desirable linear relation has been clearly shown in the plots. 
\end{remark}

    \begin{figure}[ht!]
    \centering
    \subfloat[MNIST data set]{\includegraphics[width = 0.5\textwidth]{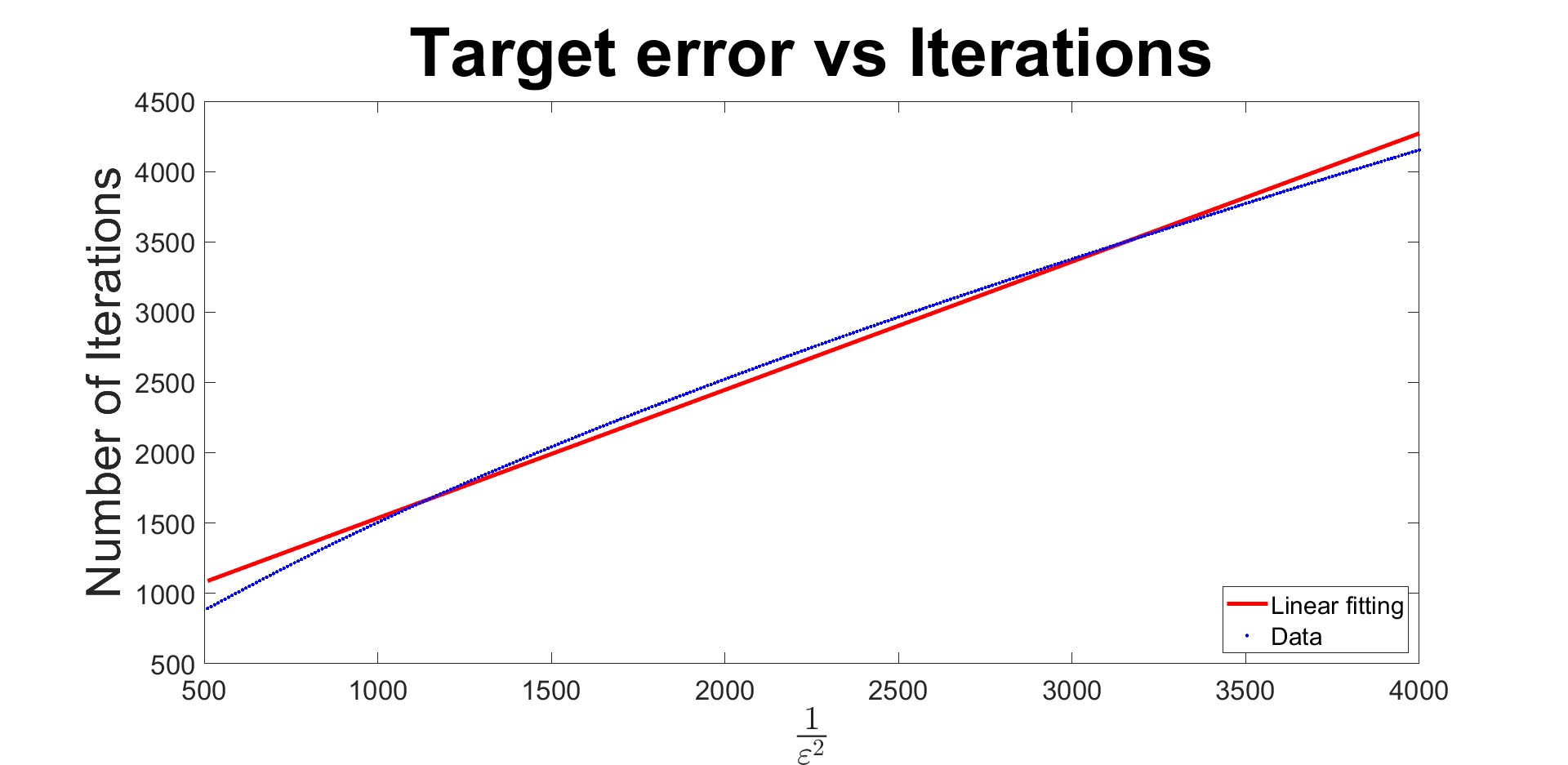}}
    \subfloat[Synthetic data set]{\includegraphics[width = 0.5\textwidth]{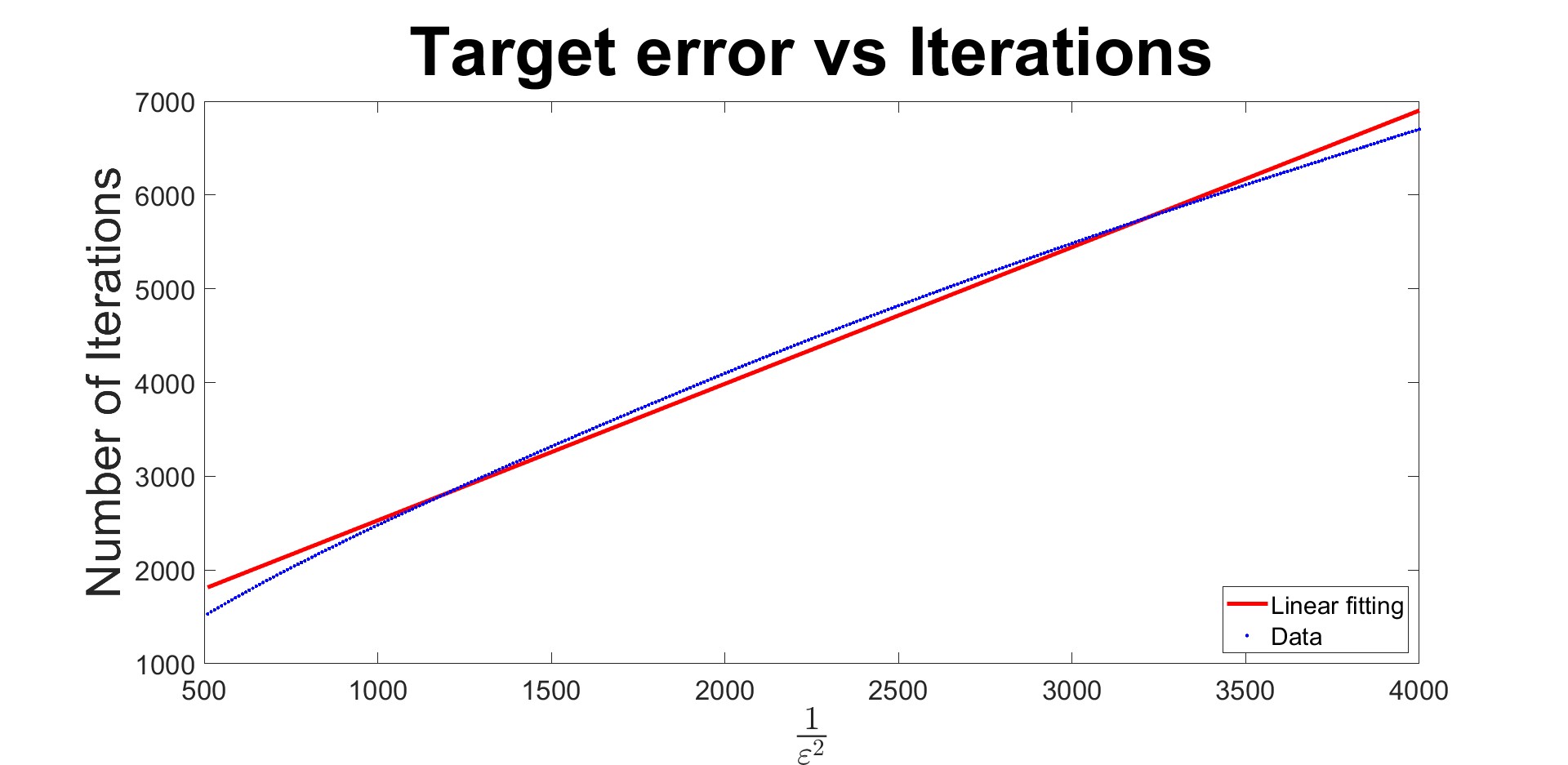}}
    \caption{Number of iteration \textit{vs} $\frac{1}{\varepsilon^2}$.}
    \label{fig:sinkhorn_eps}
\end{figure}

\section{Greenkhorn  and Its Complexity}
\begin{algorithm}[ht!]
    \renewcommand{\algorithmicrequire}{\textbf{Input:}}
    \renewcommand{\algorithmicensure}{\textbf{Output:}}
    \caption{Greenkhorn \cite{altschuler2017near}}
    \label{Alg:Greenkhorn}
    \begin{algorithmic}
        \REQUIRE probability vectors $\bfa,\bfb$, accuracy $\delta$, $\bfK\leftarrow\exp\lb-\frac{\bfC}{\gamma}\rb$
        \STATE Initialize $k\leftarrow 0$, $\bfu_0 = \bfa$, $\bfv_0 = \bfb$, $\bfP_0 \leftarrow \diag(\bfu_0) \bfK \diag(\bfv_0)$
        \REPEAT 
        \STATE $I\leftarrow \mathop{\arg\max}_i\rho(\bfa_i,(\bfP_k\ones_n)_i)$ (break the tie arbitrarily)
        \STATE $J\leftarrow \mathop{\arg\max}_j\rho(\bfb_j,(\bfP^\tran_k\ones_n)_j)$ (break the tie arbitrarily)
        \STATE $\bfu_{k+1}\leftarrow\bfu_k,\bfv_{k+1}\leftarrow\bfv_k$
        \IF{$\rho(\bfa_I,(\bfP_k\ones_n)_I)>\rho(\bfb_J,(\bfP^\tran_k\ones_n)_J)$} 
        \STATE $(\bfu_{k+1})_I\leftarrow\frac{\bfa_I}{(\bfK\bfv_k)_I}$
        \ELSE
        \STATE $(\bfv_{k+1})_J\leftarrow\frac{\bfb_J}{(\bfK^\tran\bfu_k)_J}$
        \ENDIF
        \STATE $k \leftarrow k+1$
        \STATE $\bfP_k \leftarrow \diag(\bfu_k) \bfK \diag(\bfv_k)$
        \UNTIL {$\| \bfP_k \ones_n - \bfa \|_1 + \|\bfP_k^\tran \ones_n - \bfb \|_1 \leq \delta$}
        \ENSURE $\bfP_k$.
    \end{algorithmic}
\end{algorithm}
Again without loss of generality we assume $\bfa > 0$ and $\bfb > 0$ \footnote{%
The case where there are zeros in $\bfa$ or $\bfb$ is discussed in Remark \ref{Greenkremark}.}. Recall that in each iteration of the Sinkhorn algorithm, all the elements of $\bfu_k$ or $\bfv_k$ are updated simultaneously such that the row sum of $\bfP_k$ equals $\bfa$ or the column sum of $\bfP_k$ equals $\bfb$. In contrast, only a single element of $\bfu_k$ or $\bfv_k$ is updated at a time in the Greenkhorn algorithm such that only one element of the row sum or column sum of $\bfP_k$ is equal to the target value. 
To determine which element of $\bfu_k$ or $\bfv_k$ is updated, the following scalar version of the KL divergence is used to quantify the mismatch between the elements of $\bfa$ or $\bfb$ and the corresponding elements of $\bfP\ones_n$ or $\bfP^\tran\ones_n$:
\[
\rho(x,y) := y-x+x\log\frac{x}{y},
\]
and then the one with the largest mismatch is chosen to be updated. For two probability vectors $\bfa$ and $\bfb$, it can be easily verified that 
$
KL(\bfa,\bfb) = \sum_{i=1}^n\rho(\bfa_i,\bfb_i).
$
Moreover, $\rho(x,y)$ is indeed the Bregman distance between $x$ and $y$ associated with the function $\phi(t)=t\log t$.
The Greenkhorn algorithm is described in Algorithm 
\ref{Alg:Greenkhorn}, where the initial variables  are set to be $\bfu_0 = \bfa$ and $\bfv_0 = \bfb$. Note that without further specification, $\widehat{\bfP}_k$ is the matrix that is obtained by applying Algorithm~\ref{alg:A} to the output $\bfP_k$ of the Greenkhorn algorithm in this section.
The existing complexity for the Greenkhorn algorithm is $O({\|\bfC\|^3_\infty n^2 \log n}/{\varepsilon^3})$ \cite{altschuler2017near}, while the complexity for the modified version using lifted $\bfa$ and $\bfb$ is $O({\|\bfC\|^2_\infty n^2 \log n}/{\varepsilon^2})$ \cite{lin2022efficiency}. However, a set of similar numerical experiments as in Section~\ref{sec:sink} suggest that the performance of the vanilla and modified versions are overall similar, see Figure~\ref{fig:greenkhorn}. Thus, it also makes sense to provide an improved analysis of the vanilla Greenkhorn algorithm.


\begin{figure}[ht!]
    \centering
    \subfloat[MNIST data set]{\includegraphics[width = 0.5\textwidth]{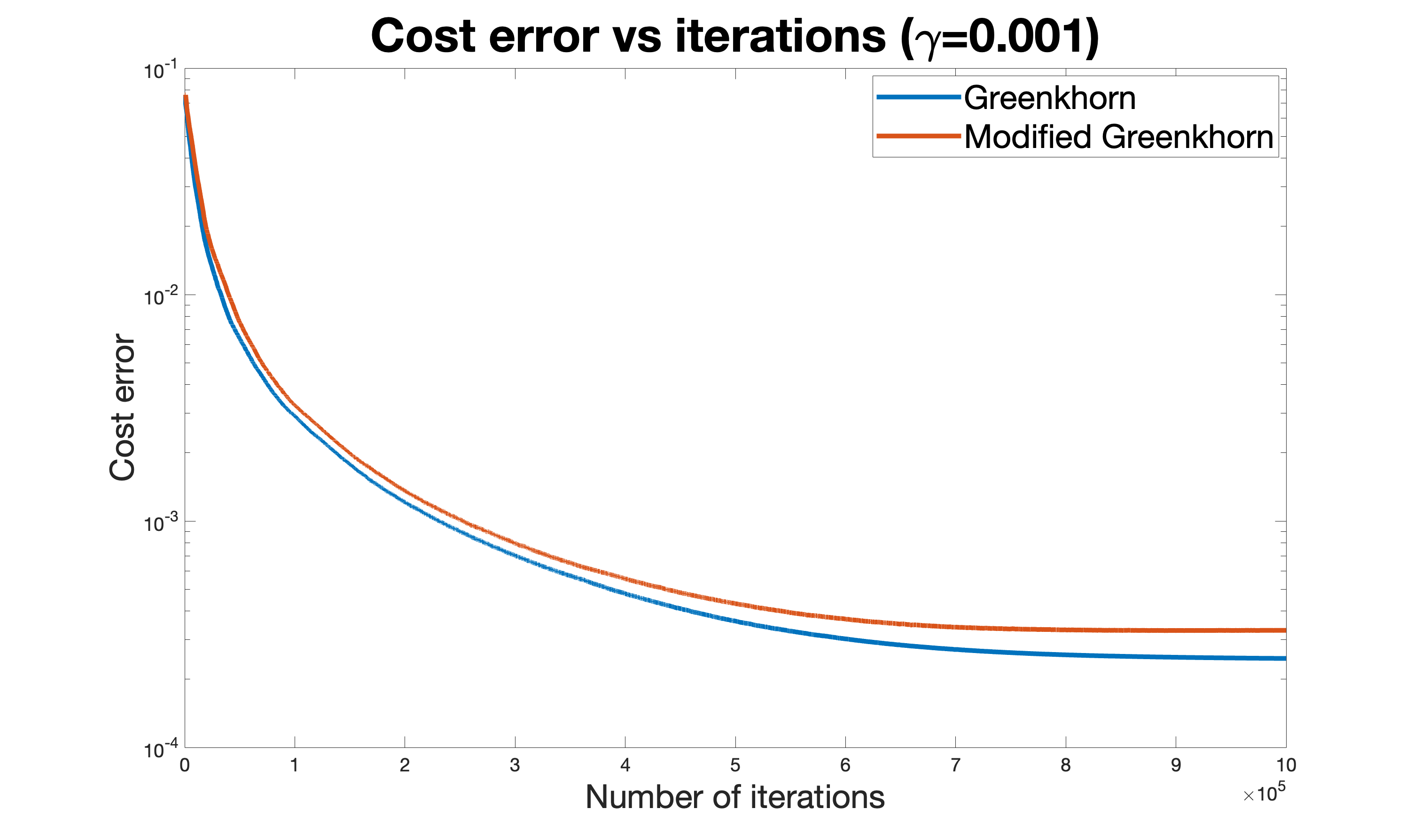}}
    \subfloat[Synthetic data set]{\includegraphics[width = 0.5\textwidth]{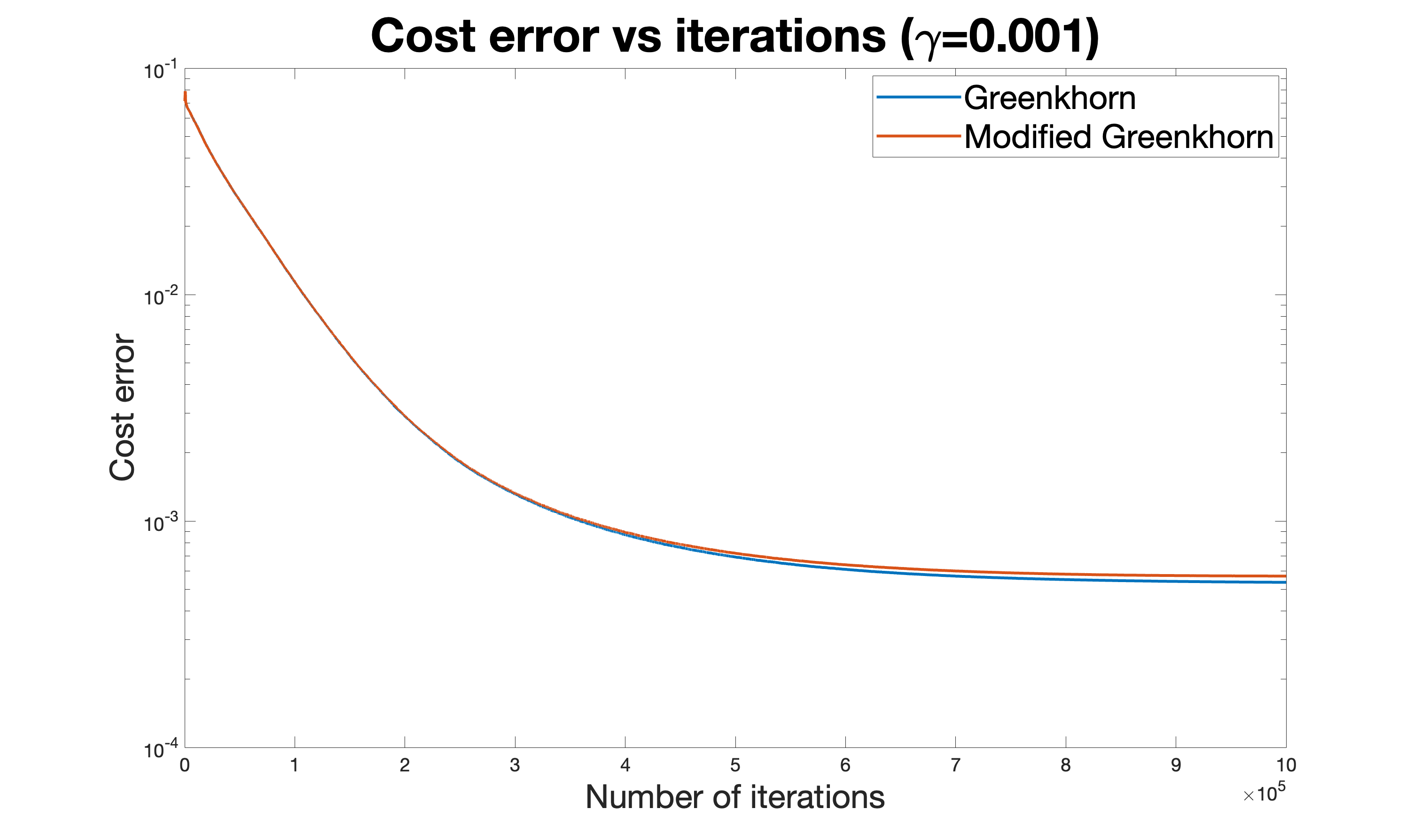}}
    \caption{Transport cost error {\it vs} number of iteration for the vanilla and modified Greenkhorn algorithms.}
    \label{fig:greenkhorn}
\end{figure}

The  convergence analysis for the Greenkhorn algorithm is overall parallel to that for the Sinkhorn algorithm, which is carried out on the dual variables:
\[
\bff_k:=\gamma\log\bfu_k \quad \text{and} \quad \bfg_k:=\gamma\log\bfv_k.
\]
That being said, a fact that should be cautioned is that the transport plan matrix $\bfP_k$ produced by the Greenkhorn algorithm may not be a probability matrix, which requires a careful treatment in the analysis. We first extend  Lemma~\ref{thm:sinkbound} to the situation where 
$\bfB$ is not necessarily a probability matrix. The proof of this lemma is presented in the appendix.
\begin{lemma}\label{G_totalbias}
    Let $\bfP^*$ be the solution to the Kantorovich problem \ref{Kan} with marginals $\bfa,\bfb$ and cost matrix $\bfC$. Let $\bfB=\rm{diag}(\bfu)\bfK\rm{diag}(\bfv)$, where $\bfu,\bfv\in\mR^n_{++}$ and $\bfK \in \R^{n\times n}$ is defined by $\bfK = \exp(\frac{\bfC}{\gamma})$. Letting $\widehat{\bfB} = \rm{Round}(\bfB,\bfa,\bfb)$, we have
    \begin{align}\label{greenkbound}
        \langle \bfC, \widehat{\bfB}\rangle \leq \langle \bfC,\bfP^* \rangle + (2+\| \bfB \ones_n - \bfa \|_1 + \|\bfB^\tran \ones_n - \bfb \|_1)\gamma\log n + 4(\| \bfB \ones_n - \bfa \|_1 + \|\bfB^\tran \ones_n - \bfb \|_1)\|\bfC\|_\infty.
    \end{align}
\end{lemma}

\begin{remark}\label{remark:tmp_ke2}
In \eqref{greenkbound}, If we set $\bfB$ to be $\bfP_k$, the output of Algorithm~\ref{Alg:Greenkhorn}, we have 
\[
\langle \bfC, \widehat{\bfP}_k\rangle \leq \langle \bfC,\bfP^* \rangle + (2+\| \bfP_k \ones_n - \bfa \|_1 + \|\bfP_k^\tran \ones_n - \bfb \|_1)\gamma\log n +4(\| \bfP_k \ones_n - \bfa \|_1 + \|\bfP_k^\tran \ones_n - \bfb \|_1)\|\bfC\|_\infty.
\]
This implies that if we set $\gamma = \frac{\varepsilon}{6\log n}$ for the regularization parameter and choose  $\delta = \min\lb1,\frac{\varepsilon}{8\|\bfC\|_\infty}\rb$ as the termination condition, the Greenkhorn algorithm will output a $\bfP_k$ such that $$\langle\bfC,\mbox{\normalfont Round}(\bfP_k,\bfa,\bfb)\rangle\leq\langle \bfC,\bfP^* \rangle+\varepsilon.$$ 
\end{remark}
The next lemma characterizes the one-iteration improvement of the Greenkhorn algorithm.
\begin{lemma}[\protect{\cite[Lemma 5, Lemma 6]{altschuler2017near}} ]\label{G_onenorm}
    For any $k\geq 0$, if $\max\lb\sum_{i=1}^n\rho(\bfa_i,(\bfP_k\ones_n)_i),\sum_{j=1}^n\rho(\bfb_j,(\bfP_k^\tran\ones_n)_j)\rb\leq 1$,
\begin{align*}
    h(\bff_{k+1},\bfg_{k+1}) - h(\bff_k,\bfg_k)\geq\frac{\gamma}{28n}(\|\bfa-\bfP_k\ones_n\|_1+\|\bfb-\bfP_k^\tran\ones_n\|_1)^2,
\end{align*}
and if $\sum_{i=1}^n\rho(\bfa_i,(\bfP_k\ones_n)_i)>1$ or $\sum_{j=1}^n\rho(\bfb_j,(\bfP_k^\tran\ones_n)_j)> 1$,
\begin{align*}
    h(\bff_{k+1},\bfg_{k+1}) - h(\bff_k,\bfg_k)\geq\frac{\gamma}{n}.
\end{align*}
\end{lemma}
The following lemma demonstrates that the infinity norm error of the iterates   in the Greenkhorn algorithm will not increase.
\begin{lemma}[\protect{\cite[Lemma 3.1]{lin2022efficiency}} ]\label{G_decreasing}
    Let $(\bff^\gamma_*,\bfg^\gamma_*)$ be any solution of \eqref{dualentro}. Then for any $k\geq 0$, we have
    \begin{align*}
        \max(\|\bff_{k+1}-\bff^\gamma_*\|_\infty,\|\bfg_{k+1}-\bfg^\gamma_*\|_\infty)\leq\max(\|\bff_k-\bff^\gamma_*\|_\infty,\|\bfg_k-\bfg^\gamma_*\|_\infty).
    \end{align*}
\end{lemma}
The proofs of Lemma~\ref{G_onenorm} and Lemma~\ref{G_decreasing} can be found in \cite{altschuler2017near} and \cite{lin2022efficiency}, respectively. To keep the work self-contained, we include the proofs in the appendix. It is worth noting that by a similar argument we can show that Lemma~\ref{G_decreasing} also holds for the Sinkhorn algorithm. The details are omitted since this result is not used in the convergence analysis of the Sinkhorn algorithm.

The following lemma provides an upper bound on the infinity norm of the optimal solution (after translation), which plays an important role in the analysis of the Greenkhorn algorithm. The proof of this lemma relies on the equicontinuity property of the optimal solution.
\begin{lemma}\label{max_f*g*}
    For the dual entropic regularized optimal transport problem \eqref{dualentro}, there exists a pair of optimal solution $(\bff_*^\gamma,\bfg_*^\gamma)$ satisfying
\begin{align}\label{max_f*g*eq}
    \max(\|\bff_*^\gamma-\gamma\log\bfa\|_\infty,\|\bfg_*^\gamma-\gamma\log\bfb\|_\infty)\leq 2\|\bfC\|_\infty.
\end{align}
\end{lemma}
\begin{proof}
   We first claim that, for any solution $\tilde{\bff}_*^\gamma, \tilde{\bfg}_*^\gamma$, there exist $(s,t),(s',t')\in[n]\times[n]$ satisfying
    \begin{align}
        &\label{cunzaist2}(\tilde{\bff}_*^\gamma)_s-\gamma\log\bfa_s+(\tilde{\bfg}_*^\gamma)_t-\gamma\log\bfb_t \geq 0,\\
        &\label{cunzaist1}(\tilde{\bff}_*^\gamma)_{s'}-\gamma\log\bfa_{s'}+(\tilde{\bfg}_*^\gamma)_{t'}-\gamma\log\bfb_{t'}\leq\|\bfC\|_\infty.
    \end{align}
    Indeed, noting that $\sum_{i,j}\bfa_i\bfb_j=1$, if for all $(s,t)\in[n]\times[n]$,
    \[
    (\tilde{\bff}_*^\gamma)_s-\gamma\log\bfa_s+(\tilde{\bfg}_*^\gamma)_t-\gamma\log\bfb_t < 0,
    \] we have
     \begin{align*}
        \sum_{i,j}\exp\lb\frac{1}{\gamma}((\tilde{\bff}_*^\gamma)_i+(\tilde{\bfg}_*^\gamma)_j-\bfC_{i,j})\rb = \sum_{i,j}\exp\lb\frac{1}{\gamma}((\tilde{\bff}_*^\gamma)_i-\gamma\log\bfa_i+(\tilde{\bfg}_*^\gamma)_j-\gamma\log\bfb_j-\bfC_{i,j})\rb \bfa_i\bfb_j < 1,
    \end{align*}
    which contradicts the fact that $\diag\lb\exp(\frac{\tilde{\bff}_*^\gamma}{\gamma})\rb\bfK\diag\lb\exp(\frac{\tilde{\bfg}_*^\gamma}{\gamma})\rb$ is a probability matrix. Thus, there exists $(s,t)$ such that \eqref{cunzaist2} holds. Similarly, we can show that there exists $(s',t')$ such that \eqref{cunzaist1} holds.
    
    Note that $h(\tilde{\bff}_*^\gamma,\tilde{\bfg}_*^\gamma) = h(\tilde{\bff}_*^\gamma+\eta\ones_n,\tilde{\bfg}_*^\gamma-\eta\ones_n)$ for any $\eta\in\mR$. Therefore, we could choose a $\eta$ such that $(\tilde{\bff}_*^\gamma)_s - \gamma \log \bfa_s +\eta = \|\bfC\|_\infty$. Set $\bff_*^\gamma = \tilde{\bff}_*^\gamma + \eta\ones_n$ and $\bfg_*^\gamma = \tilde{\bfg}_*^\gamma - \eta\ones_n$. It is evident that $(\bff_*^\gamma)_s - \gamma\log \bfa_s = \|\bfC\|_\infty$. Thus the application of
    Corollary \ref{maxf_c} implies
    \begin{align*}
        \|\bff_*^\gamma-\gamma\log \bfa\|_\infty\leq 2\|\bfC\|_\infty.
    \end{align*}
    and
    \begin{align*}
        (\bff_*^\gamma)_{s'}-\gamma\log\bfa_{s'}\geq 0.
    \end{align*}
    Since $(\bff_*^\gamma)_s - \gamma\log \bfa_s = \|\bfC\|_\infty$ and \eqref{cunzaist2} holds, we have
     \begin{align*}
        (\bfg_*^\gamma)_t-\gamma\log\bfb_t\geq-\|\bfC\|_\infty.
    \end{align*}
    In addition, it follows from  \eqref{cunzaist1} that
    \begin{align*}
        (\bfg_*^\gamma)_{t'}-\gamma\log\bfb_{t'}\leq\|\bfC\|_\infty.
    \end{align*}
    Thus, applying Corollary \ref{maxf_c} again yields that
     \begin{align*}
        \|\bfg_*^\gamma-\gamma\log \bfb\|_\infty\leq 2\|\bfC\|_\infty
    \end{align*}
    which concludes the proof.
\end{proof}

The following two lemmas can be obtained based on the above lemmas. Noting $h(\bff_*^\gamma,\bfg_*^\gamma)$ is the same for all the optimal solutions, without loss of generality we assume  $(\bff_*^\gamma, \bfg_*^\gamma)$ used in the proof   satisfies \eqref{max_f*g*eq}. 
\begin{lemma}\label{G_Tk}
    Let $(\bff_*^\gamma,\bfg_*^\gamma)$ be any solution of \eqref{dualentro}, and define  $T_k := h(\bff_*^\gamma,\bfg_*^\gamma) - h(\bff_k,\bfg_k)$. Let $(\bff_0,\bfg_0)=(\gamma\log\bfa,\gamma\log\bfb)$. We have
    \begin{align*}
       T_k = h(\bff_*^\gamma,\bfg_*^\gamma) - h(\bff_k,\bfg_k)\leq 2\|\bfC\|_\infty(\|\bfa-\bfP_k\ones_n\|_1+\|\bfb - \bfP_k^\tran\ones_n\|_1).
    \end{align*}
\end{lemma}
\begin{proof}
    Since $h(\bff,\bfg)$ is concave, we have
    \begin{align}\label{G_biasproof}
         T_k=h(\bff_*^\gamma,\bfg_*^\gamma) - h(\bff_k,\bfg_k)&\leq\langle\bff_*^\gamma-\bff_k,\bfa-\bfP_k\ones_n\rangle+\langle\bfg_*^\gamma-\bfg_k,\bfb-\bfP_k^\tran\ones_n\rangle\nonumber\\
        &\leq\|\bff_*^\gamma-\bff_k\|_\infty\|\bfa-\bfP_k\ones_n\|_1 + \|\bfg_*^\gamma-\bfg_k\|_\infty\|\bfb-\bfP_k^\tran\ones_n\|_1\nonumber\\
        &\leq\max(\|\bff_*^\gamma-\bff_k\|_\infty,\|\bfg_*^\gamma-\bfg_k\|_\infty)(\|\bfa-\bfP_k\ones_n\|_1+\|\bfb - \bfP_k^\tran\ones_n\|_1)\nonumber\\
        &\leq \max(\|\bff_*^\gamma-\bff_0\|_\infty,\|\bfg_*^\gamma-\bfg_0\|_\infty)(\|\bfa-\bfP_k\ones_n\|_1+\|\bfb - \bfP_k^\tran\ones_n\|_1)\nonumber\\
        &=\max(\|\bff_*^\gamma-\gamma\log \bfa\|_\infty,\|\bfg_*^\gamma-\gamma\log \bfb\|_\infty)(\|\bfa-\bfP_k\ones_n\|_1+\|\bfb - \bfP_k^\tran\ones_n\|_1)\\
        & \leq 2\|\bfC\|_\infty(\|\bfa-\bfP_k\ones_n\|_1+\|\bfb - \bfP_k^\tran\ones_n\|_1),
    \end{align}
    where the fourth inequality follows from Lemma~\ref{G_decreasing} and the last  inequality follows from Lemma~\ref{max_f*g*}.
\end{proof}
\begin{lemma}\label{T0}
    Let $(\bff_*^\gamma,\bfg_*^\gamma)$ be any solution of \eqref{dualentro} and let $(\bff_0,\bfg_0)=(\gamma\log\bfa,\gamma\log\bfb)$. Then
    \begin{align*}
        h(\bff_*^\gamma,\bfg_*^\gamma)-h(\bff_0,\bfg_0)=h(\bff_*^\gamma,\bfg_*^\gamma)-h(\gamma\log\bfa,\gamma\log\bfb)\leq 4\|\bfC\|_\infty.
    \end{align*}
\end{lemma}
\begin{proof}
    By the definition of function $h$, we have
    \begin{align*}
        h(\bff_*^\gamma,\bfg_*^\gamma)-h(\gamma\log\bfa,\gamma\log\bfb) &= \la\bff_*^\gamma-\gamma\log\bfa,\bfa\ra + \la\bfg_*^\gamma-\gamma\log\bfb,\bfb\ra\\
        &-\gamma\sum_{i,j}\exp\lb\frac{1}{\gamma}((\bff_*^\gamma)_i+(\bfg_*^\gamma)_j-\bfC_{i,j})\rb + \gamma \sum_{i,j}\exp\lb\frac{1}{\gamma}(\gamma\log\bfa_i+\gamma\log\bfb_j-\bfC_{i,j})\rb\\
        &=\la\bff_*^\gamma-\gamma\log\bfa,\bfa\ra + \la\bfg_*^\gamma-\gamma\log\bfb,\bfb\ra-\gamma + \gamma\sum_{i,j}\bfa_i\bfb_j\exp\lb-\frac{1}{\gamma}\bfC_{i,j}\rb\\
        &\leq\la\bff_*^\gamma-\gamma\log\bfa,\bfa\ra + \la\bfg_*^\gamma-\gamma\log\bfb,\bfb\ra\\
        &\leq \|\bff_*^\gamma-\gamma\log\bfa\|_\infty\|\bfa\|_1+\|\bfg_*^\gamma-\gamma\log\bfb\|_\infty\|\bfb\|_1\\
        &=\|\bff_*^\gamma-\gamma\log\bfa\|_\infty+\|\bfg_*^\gamma-\gamma\log\bfb\|_\infty\\
        &=2\max(\|\bff_*^\gamma-\gamma\log\bfa\|_\infty,\|\bfg_*^\gamma-\gamma\log\bfb\|_\infty)\\
        &\leq  2\|\bfC\|_\infty,
    \end{align*}
    where the last inequality follows from Lemma~\ref{max_f*g*}.
\end{proof}

Now we are ready to prove the main theorem about the computational complexity of the Greenkhorn algorithm. Note that, compared with the proof for Theorem~\ref{Sink_theorem}, we need to discuss the two different cases of the iterates very carefully.
\begin{theorem}\label{G_iternumber}
    Algorithm \ref{Alg:Greenkhorn} with initialization $(\bfu_0,\bfv_0)=(\bfa,\bfb)$ outputs a matrix $\bfP_k$ satisfying $\|\bfa - \bfP_k\ones_n\|_1 + \|\bfb - \bfP_k^\tran\ones_n \|_1 \leq \delta$ in less than
    $
        2\lceil\frac{56n\|\bfC\|_\infty}{\gamma\delta}\rceil  + 2\lceil\frac{4n\|\bfC\|_\infty}{\gamma}\rceil
    $
    number of iterations. 
    Furthermore, the computational complexity for the algorithm to achieve an $\varepsilon$-accuracy after rounding is \begin{align*}O\left( \max\lb\frac{\|\bfC\|^2_\infty n^2 \log n}{\varepsilon^2},\frac{\|\bfC\|_\infty n^2 \log n}{\varepsilon}\rb\right)=O\left(\frac{\|\bfC\|^2_\infty n^2 \log n}{\varepsilon^2}\right)
    \end{align*}
    when $\varepsilon$ is sufficiently small.
\end{theorem}
\begin{proof}
    Define 
    \begin{align*}
        &S_1=\{k\in N:\sum_{i=1}^n\rho(\bfa_i,(\bfP_k\ones_n)_i)> 1 \text{ or }\sum_{j=1}^n\rho(\bfb_j,(\bfP_k^\tran\ones_n)_j)> 1\},\\
        &S_2=N\backslash S_1.
    \end{align*}
    When  $k\in S_1$, by Lemma \ref{G_onenorm}, 
    \begin{align*}
        T_k-T_{k+1}=h(\bff_{k+1},\bfg_{k+1})-h(\bff_k,\bfg_k)\geq \frac{\gamma}{n}.
    \end{align*}
    On the other hand, when $k\in S_2$, by Lemma \ref{G_onenorm} and Lemma \ref{G_Tk}
    \begin{align*}
        T_k-T_{k+1}=h(\bff_{k+1},\bfg_{k+1}) - h(\bff_k,\bfg_k)\geq\frac{\gamma}{28n}(\|\bfa-\bfP_k\ones_n\|_1+\|\bfb-\bfP_k^\tran\ones_n\|_1)^2\geq \frac{\gamma}{112n\|\bfC\|_\infty^2}T_k^2.
    \end{align*}
    This means that $\{T_k\}_k$ is a decreasing sequence. So by Lemma \ref{T0}, we have
    \begin{align*}
        |S_1|\leq\frac{h(\bff^\gamma_*,\bfg_*^\gamma)-h(\bff_0,\bfg_0)}{\frac{\gamma}{n}}\leq\frac{4n\|\bfC\|_\infty}{\gamma}.
    \end{align*}
    Noting that $\lcb\frac{1}{T_k}\rcb_k$ is an increasing sequence,  if $k\in S_2$,
    \begin{align}\label{S1update}
        \frac{1}{T_{k+1}}-\frac{1}{T_k} = \frac{T_k-T_{k+1}}{T_{k+1}T_k}\geq\frac{T_k-T_{k+1}}{T_k^2}\geq\frac{\gamma}{112n\|\bfC\|_\infty^2}.
    \end{align}
   Setting $k_0=\lceil\frac{56n\|\bfC\|_\infty}{\delta\gamma}\rceil+|S_1|$, one has 
    \begin{align*}
        \frac{1}{T_{k_0}}\geq \frac{1}{T_{k_0}}-\frac{1}{T_0}\geq \sum_{k\in S_2\mbox{ and } k<k_0} \left(\frac{1}{T_{k+1}}-\frac{1}{T_k}\right)\geq \frac{1}{2\|\bfC\|_\infty\delta},
    \end{align*}
    since there are at least $\lceil\frac{56n\|\bfC\|_\infty}{\delta\gamma}\rceil$ terms in the sum. Since $T_{k_0}\leq 2\|\bfC\|_\infty\delta$, applying Lemma \ref{G_onenorm} again shows that the algorithm will terminate within at most 
    $
        \lceil 2\|\bfC\|_\infty\delta / \frac{\gamma\delta^2}{28n}\rceil  + |S_1| = \lceil\frac{56n\|\bfC\|_\infty}{\gamma\delta}\rceil + |S_1|
    $
    additional number of iterations. Therefore, the total number of iterations required by the Greenkhorn algorithm  to terminate is smaller than
    \begin{align*}
        k_0+\lceil\frac{56n\|\bfC\|_\infty}{\gamma\delta}\rceil + |S_1| \leq 2\lceil\frac{56n\|\bfC\|_\infty}{\gamma\delta}\rceil  + 2\lceil\frac{4n\|\bfC\|_\infty}{\gamma}\rceil.
    \end{align*}
    Noting Remark~\ref{remark:tmp_ke2}, since the per iteration computational complexity of the Greenkhorn algorithm is $O(n)$, we can finally conclude that the total computational cost  for the Greenkhorn algorithm to achieve an $\varepsilon$-accuracy after rounding is $O\left( \max\lb\frac{\|\bfC\|^2_\infty n^2 \log n}{\varepsilon^2},\frac{\|\bfC\|_\infty n^2 \log n}{\varepsilon}\rb\right)$ by setting $\gamma = \frac{\varepsilon}{6\log n}$ and $\delta = \min(1,\frac{\varepsilon}{8\|\bfC\|_\infty})$.
\end{proof}
\begin{remark}\label{Greenkremark}
 For the case when there are zeros in $\bfa$ and $\bfb$, due to our selection of $\bfu_0 = \bfa$ and $\bfv_0 = \bfb$, if we use the convention $\rho(0,0)=0$, the Greenkhorn algorithm never updates zeros in $\bfu$ and $\bfv$ in each iteration. This means that the rows with index in $A_1$ and the columns with index in $A_2$ of each $\bfP_k$ are all zeros, where $A_1$ and $A_2$ is defined in the same way as in Remark~\ref{remarkSink}. Thus we only need to focus on the nonzero part of $\{\bfP_k\}_k$, which is identical to the sequence generated by applying the Greenkhorn algorithm to solve \eqref{entroKantilde}. Therefore, the complexity is  still  $O\left( \max\lb\frac{\|\bfC\|^2_\infty n^2 \log n}{\varepsilon^2},\frac{\|\bfC\|_\infty n^2 \log n}{\varepsilon}\rb\right)$ in this case.
\end{remark}

\begin{remark}
    Similar to Section~\ref{sec:sink}, numerical experiments have been conducted to verify the linear dependency of the complexity on ${1}/{\varepsilon^2}$, see Figure~\ref{fig:greenkhorn_eps}.
\end{remark}

\begin{figure}[ht!]
    \centering
    \subfloat[MNIST data set]{\includegraphics[width = 0.5\textwidth]{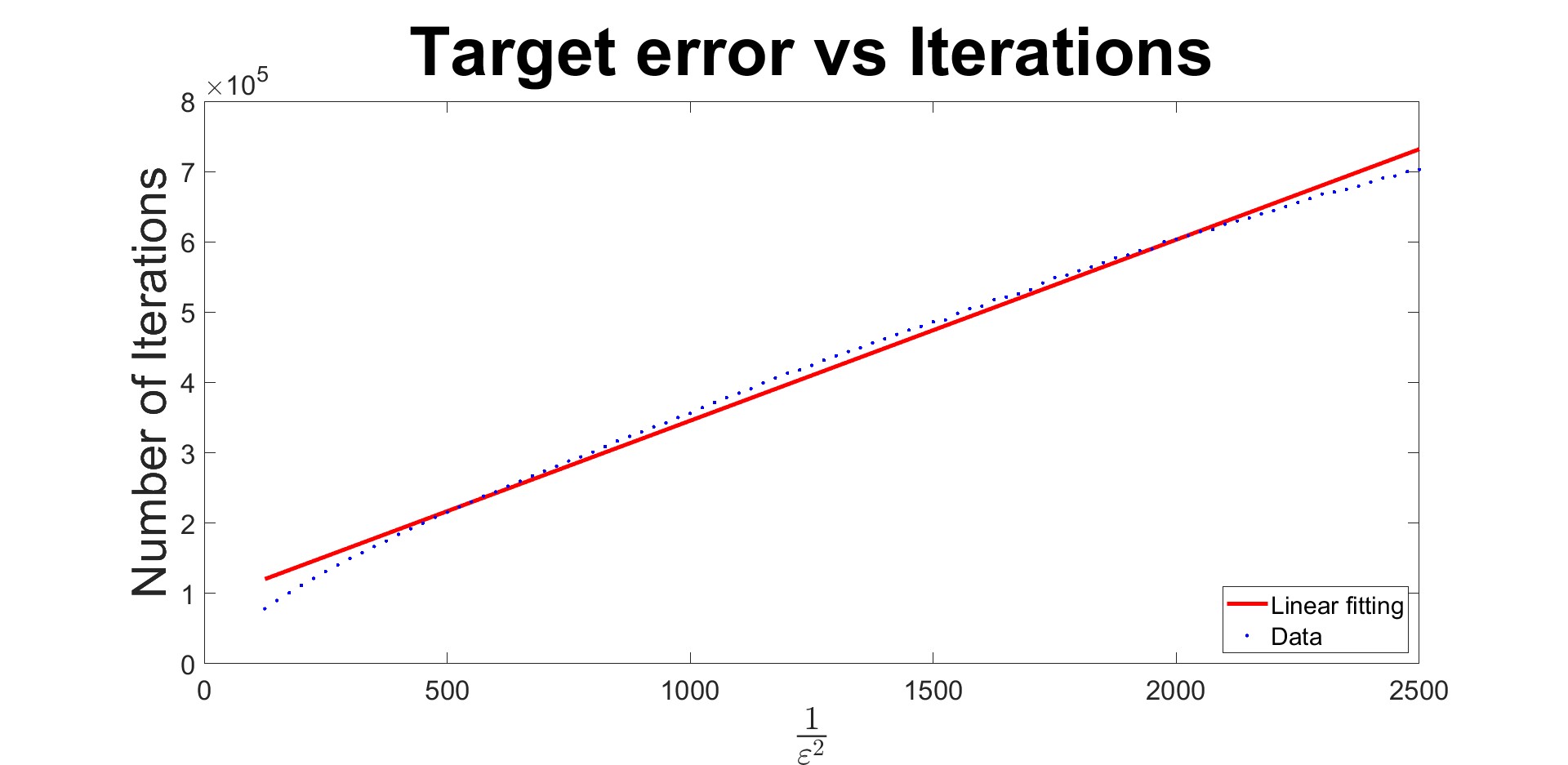}}
    \subfloat[Synthetic data set]{\includegraphics[width = 0.5\textwidth]{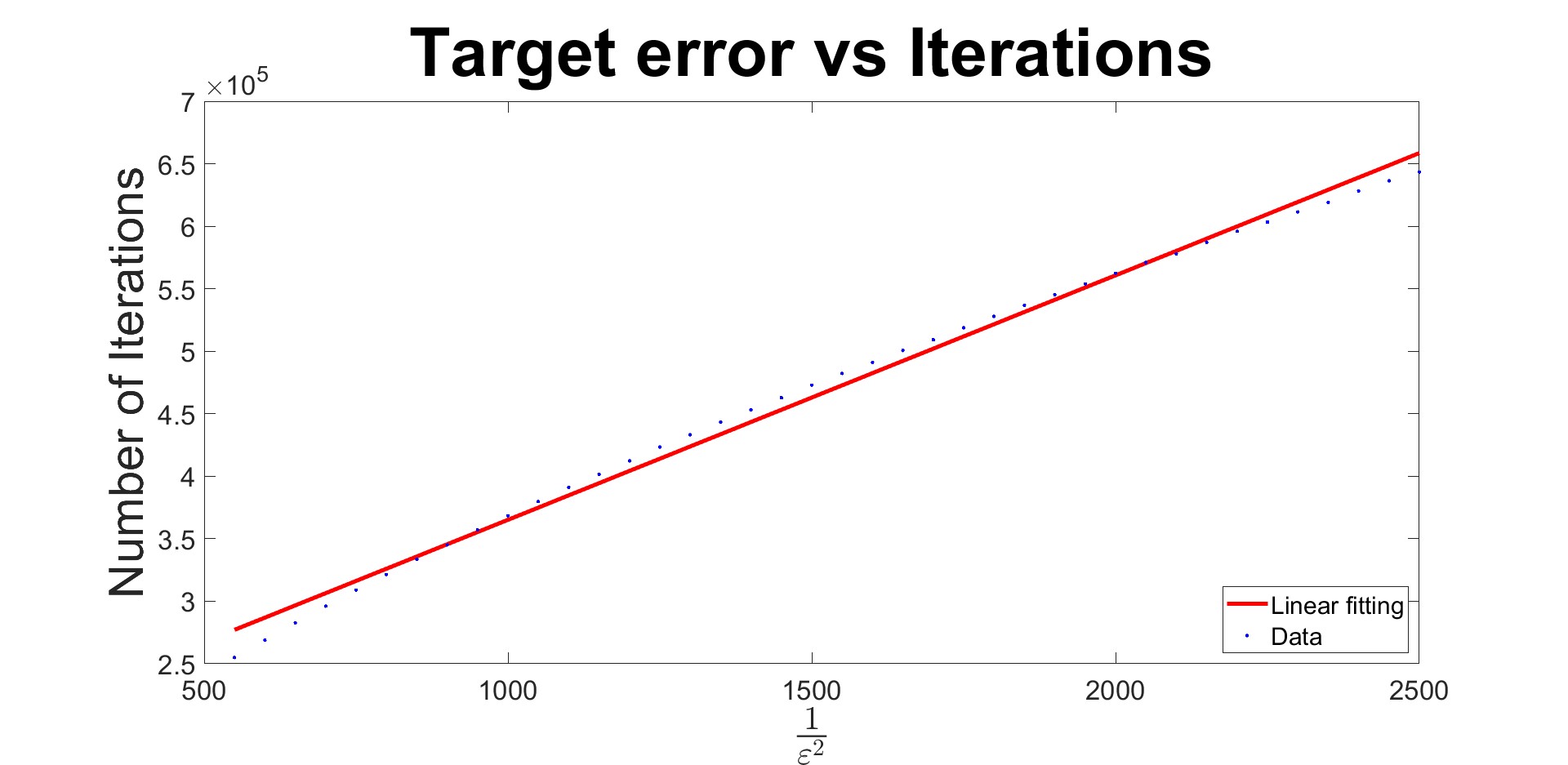}}
    \caption{Number of iteration \textit{vs} $\frac{1}{\varepsilon^2}$.}
    \label{fig:greenkhorn_eps}
\end{figure}

\begin{remark}
Though our work focuses on the complexity analysis of the Sinkhorn and Greenkhorn algorithms, it is worth noting that there are also other algorithms to solve the OT problem. For example, a set of first-order methods (APDAGD \cite{pmlr-v80-dvurechensky18a};APDAMD \cite{lin2022efficiency}; AAM \cite{AAM}; HPD \cite{HPD}) using primal-dual gradient descent or mirror descent schemes have been shown to enjoy an $\tilde{O}({n^{2.5}}/{\varepsilon})$ complexity. Note that the technique of modifying $\bfa, \bfb$ is also used in APDAGD \cite{pmlr-v80-dvurechensky18a} and APDAMD \cite{lin2022efficiency} and our analysis indeed suggests  that the same computational complexity can  be achieved without modifying $\bfa,\bfb$. 
 Algorithms with theoretical runtime $\tilde{O}({n^2}/{\varepsilon})$ have been developed in \cite{blanchet2020optimal} and \cite{quanrud2018approximating}, although practical implementations of these algorithms remain unavailable. In \cite{dual-extra,extragradient}, more practical algorithms with the  $\tilde{O}({n^2}/{\varepsilon})$ complexity have been developed.
\end{remark}

\section{Conclusion}
We first show that the computational complexity of the vanilla Sinkhorn algorithm with any feasible initial vectors is $O(\frac{n^2\|\bfC\|_\infty^2\log n}{\varepsilon^2})$. For the vanilla Greenhorn algorithm, a careful analysis shows that if the initialization is chosen to be the target probability vectors, the computational complexity of the algorithm is also $O(\frac{n^2\|\bfC\|_\infty^2\log n}{\varepsilon^2})$. The equicontinuity property of the discrete entropic regularized optimal transport problem plays a key role in the analysis.

\bibliographystyle{unsrt}
\bibliography{ref}
\appendix
\section{Supplementary Proofs for Lemmas~\ref{round_error}, \ref{thm:sinkbound}, and \ref{deltaimprove}}
\begin{proof}[Proof of Lemma \ref{round_error}]
It is clear that $\bfP''\leq\bfP'\leq\bfP$, $\|\Delta_{\bfa}\|_1=\|\Delta_{\bfb}\|_1$.
Let $\widehat{\bfP} = \rm{Round}(\bfP,\bfa,\bfb)$. It is easy to verify that $\widehat{\bfP}\ones_n=\bfa$ and $\widehat{\bfP}^\tran\ones_n=\bfb$. Moreover,
\begin{align*}
    \|\bfP-\widehat{\bfP}\|_1&\leq\|\bfP-\bfP''\|_1+\frac{1}{\|\Delta_{\bfa}\|_1}\|\Delta_{\bfa}\Delta_{\bfb}^\tran\|_1= \|\bfP-\bfP''\|_1 + \|\bfa\|_1 - \|\bfP''\|_1 \\
    &= 2\|\bfP-\bfP''\|_1 + \|\bfa\|_1 - \|\bfP\|_1 
    = 2\lb\|\bfP-\bfP'\|_1+\|\bfP'-\bfP''\|_1\rb + \|\bfa\|_1 - \|\bfP\|_1,
\end{align*}
where the last equality arises from $\bfP\geq\bfP'\geq\bfP''$.  Since
\begin{align*}
    \|\bfP-\bfP'\|_1 &= \sum_{i=1}^n\lsb(\bfP\ones_n)_i-\bfa_i\rsb_+=\frac{1}{2}\sum_{i=1}^n\lsb(\bfP\ones_n)_i-\bfa_i + |(\bfP\ones_n)_i-\bfa_i|\rsb  = \frac{1}{2}\lb\|\bfa - \bfP\ones_n\|_1+\|\bfP\|_1-\|\bfa\|_1\rb
\end{align*}
and
\begin{align*}
    \|\bfP'-\bfP''\|_1 = \sum_{j=1}^n[(\bfP'^\tran\ones_n)_j-\bfb_j]_+ \leq \sum_{j=1}^n[(\bfP^\tran\ones_n)_j-\bfb_j]_+\leq \|\bfb-\bfP^\tran\ones_n\|_1,
\end{align*}
we have
\begin{align*}
    \|\bfP-\widehat{\bfP}\|_1&\leq 2\lb\|\bfP-\bfP'\|_1+\|\bfP'-\bfP''\|_1\rb + \|\bfa\|_1 - \|\bfP\|_1\\
    &\leq \|\bfa - \bfP\ones_n\|_1 + 2\|\bfb-\bfP^\tran\ones_n\|_1\leq 2\lb\|\bfa-\bfP\ones_n\|_1+\|\bfb - \bfP^\tran\ones_n\|_1\rb,
\end{align*}
which completes the proof.
\end{proof}

\begin{proof}[Proof of Lemma \ref{thm:sinkbound}]
     Let $\widehat{\bfB}$ be the output of Round$(\bfB,\bfa,\bfb)$ and $\widehat{\bfP}$ be the output of Round$(\bfP^*,\bfB\ones_n,\bfB^\tran\ones_n)$.  By verifying the KKT condition, we have
     \begin{align*}
        \bfB\in\mathop{\arg\min}_{\bfP\in\bfPi(\bfB\ones_n,\bfB^\tran\ones_n)}\la\bfC,\bfP\ra +\gamma\sum_{i,j}\bfP_{i,j}(\log \bfP_{i,j}-1).\numberthis\label{eq:tmp_ke02}
    \end{align*}
    Thus,
 \begin{align*}
        \langle\bfC,\widehat{\bfB}\rangle &\leq \langle\bfC,\bfB\rangle + \|\bfC\|_\infty\|\bfB-\widehat{\bfB}\|_1\\
        &\leq\langle\bfC,\widehat{\bfP}\rangle+\gamma\sum_{i,j}(\widehat{\bfP})_{i,j}(\log (\widehat{\bfP})_{i,j}-1)-\gamma\sum_{i,j}(\bfB)_{i,j}(\log (\bfB)_{i,j}-1) +\|\bfC\|_\infty\|\bfB-\widehat{\bfB}\|_1 \\
        &=\langle\bfC,\widehat{\bfP}\rangle+\gamma\sum_{i,j}(\widehat{\bfP})_{i,j}\log (\widehat{\bfP})_{i,j}-\gamma\sum_{i,j}(\bfB)_{i,j}\log (\bfB)_{i,j} +\|\bfC\|_\infty\|\bfB-\widehat{\bfB}\|_1 \\
        &\leq\langle\bfC,\widehat{\bfP}\rangle-\gamma\sum_{i,j}(\bfB)_{i,j}\log (\bfB)_{i,j} +\|\bfC\|_\infty\|\bfB-\widehat{\bfB}\|_1\\
        & \leq\langle\bfC,\widehat{\bfP}\rangle+2\gamma\log n+\|\bfC\|_\infty\|\bfB-\widehat{\bfB}\|_1 \\
        &\leq\langle\bfC,\bfP^*\rangle +\|\bfC\|_\infty\|\widehat{\bfP}-\bfP^*\|_1 +2\gamma\log n+\|\bfC\|_\infty\|\bfB-\widehat{\bfB}\|_1.
\end{align*}
Then the proof is complete after applying   Lemma \ref{round_error} to bound $\|\widehat{\bfP}-\bfP^*\|_1$ and $\|\bfB-\widehat{\bfB}\|_1$.
\end{proof}

\begin{proof}[Proof of Lemma~\ref{deltaimprove}]
    Assume $k\geq 2$ and $k$ is even. According to the update rule of the Sinkhorn algorithm, we have
    \[
    \bff_{k+1} = \gamma \log\bfa - \gamma \log \bfK e^{\frac{\bfg_k}{\gamma}},\quad \bfg_{k+1} = \bfg_k
    \]
    and $\sum_{i,j}\exp\lb\frac{1}{\gamma}((\bff_k)_i + (\bfg_k)_j - \bfC_{ij})\rb =\sum_{i,j}\exp\lb\frac{1}{\gamma}((\bff_{k+1})_i + (\bfg_{k+1})_j - \bfC_{ij})\rb = 1$. Thus
    \begin{align*}
        h(\bff_{k+1},\bfg_{k+1}) - h(\bff_k,\bfg_k) &= \langle \bfa, \bff_{k+1} - \bff_k \rangle= \langle \bfa, \gamma \log \frac{\bfa}{\bfK e^{\bfg_k/{\gamma}}} - \gamma\log e^{\frac{\bff_k}{\gamma}} \rangle\\
        & = \langle \bfa, \gamma \log \frac{\bfa}{ e^{\bff_k/{\gamma}}\odot\bfK e^{\bfg_k/{\gamma}}
         }\rangle = \gamma \langle \bfa, \log \frac{\bfa}{\bfP_k\ones_n }\rangle = \gamma KL(\bfa\|\bfP_k\ones_n)\\
         &\geq\frac{\gamma}{2}\|\bfa-\bfP_k\ones_n\|_1^2 = \frac{\gamma}{2}(\|\bfa-\bfP_k\ones_n\|_1+\|\bfb-\bfP_k^\tran\ones_n\|_1)^2,
    \end{align*}
    where the inequality utilizes Pinsker's inequality and the last equality follows from the fact that $\bfP_k^\tran\ones_n = \bfb$ if $k
    \geq 2$ and $k$ is even. Then case when $k$ is odd can be  proved in a similar manner.
\end{proof}

\section{Proof of Lemma~\ref{G_totalbias}}
    Let $\widehat{\bfB}=\mbox{Round}(\bfB,\bfa,\bfb)$ and $\widehat{\bfP}=\mbox{Round}(\bfP^*,\bfB\ones_n,\bfB^\tran\ones_n)$.  
    Note that $\|\bfB\|_1=\|\widehat{\bfP}\|_1$. Utilizing \eqref{eq:tmp_ke02} again yields that
 \begin{align*}
        \langle\bfC,\widehat{\bfB}\rangle &\leq \langle\bfC,\bfB\rangle + \|\bfC\|_\infty\|\bfB-\widehat{\bfB}\|_1\nonumber\\
        &\leq\langle\bfC,\widehat{\bfP}\rangle+\gamma\sum_{i,j}(\widehat{\bfP})_{i,j}(\log (\widehat{\bfP})_{i,j}-1)-\gamma\sum_{i,j}(\bfB)_{i,j}(\log (\bfB)_{i,j}-1) +\|\bfC\|_\infty\|\bfB-\widehat{\bfB}\|_1 \nonumber\\
        &=\langle\bfC,\widehat{\bfP}\rangle+\gamma\|\bfB\|_1\sum_{i,j}\frac{(\widehat{\bfP})_{i,j}}{\|\bfB\|_1}\log \frac{(\widehat{\bfP})_{i,j}}{\|\bfB\|_1}-\gamma\|\bfB\|_1\sum_{i,j}\frac{(\bfB)_{i,j}}{\|\bfB\|_1}\log \frac{(\bfB)_{i,j}}{\|\bfB\|_1} +\|\bfC\|_\infty\|\bfB-\widehat{\bfB}\|_1\nonumber \\
        &\leq\langle\bfC,\widehat{\bfP}\rangle-\gamma\|\bfB\|_1\sum_{i,j}\frac{(\bfB)_{i,j}}{\|\bfB\|_1}\log \frac{(\bfB)_{i,j}}{\|\bfB\|_1} +\|\bfC\|_\infty\|\bfB-\widehat{\bfB}\|_1\nonumber\\
        & \leq\langle\bfC,\widehat{\bfP}\rangle+2\|\bfB\|_1\gamma\log n+\|\bfC\|_\infty\|\bfB-\widehat{\bfB}\|_1\nonumber \\
        &\leq\langle\bfC,\bfP^*\rangle +\|\bfC\|_\infty\|\widehat{\bfP}-\bfP^*\|_1 +2\|\bfB\|_1\gamma\log n+\|\bfC\|_\infty\|\bfB-\widehat{\bfB}\|_1,
\end{align*}
    Note that $\|\bfB\|_1$ can be bounded as follows:
    \begin{align*}
        &\|\bfB\|_1= \|\bfB\ones_n\|_1\leq\|\bfa\|_1+\|\bfB\ones_n-\bfa\|_1 = 1+\|\bfB\ones_n-\bfa\|_1,\nonumber\\
        &\|\bfB\|_1= \|\bfB^\tran\ones_n\|_1\leq\|\bfb\|_1+\|\bfB^\tran\ones_n-\bfb\|_1 = 1+\|\bfB^\tran\ones_n-\bfb\|_1.
    \end{align*}
    We can conclude the proof by applying   Lemma \ref{round_error} to bound $\|\widehat{\bfP}-\bfP^*\|_1$ and $\|\bfB-\widehat{\bfB}\|_1$.

\section{Supplementary Proofs for Lemmas~\ref{G_onenorm} and \ref{G_decreasing}}
\begin{proof}[Proof of Lemma~\ref{G_onenorm}]
    We first establish the following result which can be viewed as a generalization of the Pinsker's inequality.
    For any $\bfx\in\Delta_n,\bfy\in\mR_{++}^n$, if $\sum_{i=1}^n\rho(\bfx_i,\bfy_i)\leq 1$, we have
    \begin{align*}
        \|\bfx-\bfy\|_1^2\leq7\sum_{i=1}^n\rho(\bfx_i,\bfy_i).\numberthis\label{eq:tmp_ke03}
    \end{align*}
    The proof of this result is as follows. Letting $\bar{\bfy}:=\frac{\bfy}{\|\bfy\|_1}$, we have
    \begin{align*}
        \sum_{i=1}^n\rho(\bfx_i,\bfy_i) &= \sum_{i=1}^n(\bfy_i-\bfx_i)+\bfx_i\log\frac{\bfx_i}{\bfy_i}=\|\bfy\|_1-1+\sum_{i=1}^n\bfx_i\log\frac{\bfx_i}{\|\bfy\|_1\bar{\bfy}_i}\\
        &=\|\bfy\|_1-1-\log\|\bfy\|_1\sum_{i=1}^n\bfx_i+KL(\bfx\|\bar{\bfy})=\|\bfy\|_1-1-\log\|\bfy\|_1+KL(\bfx\|\bar{\bfy}).
    \end{align*}
    When $\sum_{i=1}^n\rho(\bfx_i,\bfy_i)\leq 1$, it is easy to see $\|\bfy\|_1-1-\log\|\bfy\|_1\leq 1$. Moreover, it  can be verified directly that $\|\bfy\|_1-1-\log\|\bfy\|_1\geq \frac{1}{5}(\|\bfy\|_1-1)^2$ if $\|\bfy\|_1-1-\log\|\bfy\|_1\leq 1$. Applying the Pinsker's inequality yields
    \begin{align*}
        \sum_{i=1}^n\rho(\bfx_i,\bfy_i) &=\|\bfy\|_1-1-\log\|\bfy\|_1+KL(\bfx\|\bar{\bfy})\\
        &\geq \frac{1}{5}(\|\bfy\|_1-1)^2+\frac{1}{2}\|\bfx-\bar{\bfy}\|_1^2.
    \end{align*}
    It follows that
    \begin{align*}
        \|\bfx-\bfy\|_1^2&\leq(\|\bar{\bfy}-\bfy\|_1+\|\bfx-\bar{\bfy}\|_1)^2 = (|\|\bfy\|_1-1|+\|\bfx-\bar{\bfy}\|_1)^2\\
        &\leq \frac{7}{5}(\|\bfy\|_1-1)^2+\frac{7}{2}\|\bfx-\bar{\bfy}\|_1^2 \leq 7\sum_{i=1}^n\rho(\bfx_i,\bfy_i),
    \end{align*}
    which establishes \eqref{eq:tmp_ke03} provided $\sum_{i=1}^n\rho(\bfx_i,\bfy_i)\leq 1$.
    
    To proceed,   without loss of generality, suppose the $(k+1)$-th iteration of the Greenkhorn algorithm updates $\bff_I$. Then we have
    \begin{align*}
        h(\bff_{k+1},\bfg_{k+1})-h(\bff_k,\bfg_k) &= \bfa_I((\bff_{k+1})_I-(\bff_k)_I)\\
        &-\gamma\sum_{j=1}^n\lsb\exp\lb\frac{1}{\gamma}((\bff_{k+1})_I+(\bfg_{k+1})_j-\bfC_{I,j})\rb-\exp\lb\frac{1}{\gamma}((\bff_{k})_I+(\bfg_{k})_j-\bfC_{I,j})\rb\rsb\\
        &=\bfa_I((\bff_{k+1})_I-(\bff_k)_I)-\gamma((\bfP_{k+1}\ones_n)_I-(\bfP_{k}\ones_n)_I)\\
        &=\gamma\bfa_I\log\lb\frac{\exp\lb\frac{(\bff_{k+1})_I}{\gamma}\rb}{\exp\lb\frac{(\bff_{k})_I}{\gamma}\rb}\rb -\gamma((\bfP_{k+1}\ones_n)_I-(\bfP_{k}\ones_n)_I)\\
        & = \gamma\bfa_I\log\lb\frac{\sum_{j=1}^n\exp\lb\frac{(\bff_{k+1})_I+(\bfg_{k+1})_j-\bfC_{I,j}}{\gamma}\rb}{\sum_{j=1}^n\exp\lb\frac{(\bff_{k})_I+(\bfg_{k})_j-\bfC_{I,j}}{\gamma}\rb}\rb -\gamma((\bfP_{k+1}\ones_n)_I-(\bfP_{k}\ones_n)_I)\\
        & = \gamma\bfa_I\log\lb\frac{(\bfP_{k+1}\ones_n)_I}{(\bfP_{k}\ones_n)_I}\rb -\gamma((\bfP_{k+1}\ones_n)_I-(\bfP_{k}\ones_n)_I)\\
        & = \gamma \bfa_I\log\lb\frac{\bfa_I}{(\bfP_{k}\ones_n)_I}\rb -\gamma(\bfa_I-(\bfP_{k}\ones_n)_I)\\
        &=\gamma\rho(\bfa_I,(\bfP_{k}\ones_n)_I),
    \end{align*}
    where the sixth equality follows from $(\bfP_{k+1}\ones_n)_I=\bfa_I$. If 
    \begin{align*}
        \max\lb\sum_{i=1}^n\rho(\bfa_i,(\bfP_k\ones_n)_i),\sum_{j=1}^n\rho(\bfb_j,(\bfP_k^\tran\ones_n)_j)\rb\leq 1,
    \end{align*}
    according to the update rule of the Greenkhorn algorithm and inequality \eqref{eq:tmp_ke03}, we have
    \begin{align*}
        h(\bff_{k+1},\bfg_{k+1})-h(\bff_k,\bfg_k)&=\gamma\rho(\bfa_I,(\bfP_{k}\ones_n)_I)\\
        &\geq\frac{\gamma}{2n}\lb\sum_{i=1}^n\rho(\bfa_i,(\bfP_{k}\ones_n)_i)+\sum_{j=1}^n\rho(\bfb_j,(\bfP_{k}^\tran\ones_n)_j)\rb\\
        &\geq \frac{\gamma}{14n}(\|\bfa-\bfP_{k}\ones_n\|_1^2+\|\bfb-\bfP^\tran_{k}\ones_n\|_1^2)\\
        &\geq \frac{\gamma}{28n}(\|\bfa-\bfP_{k}\ones_n\|_1+\|\bfb-\bfP^\tran_{k}\ones_n\|_1)^2.
    \end{align*}
    If $\sum_{i=1}^n\rho(\bfa_i,(\bfP_k\ones_n)_i)>1$ or $\sum_{j=1}^n\rho(\bfb_j,(\bfP_k^\tran\ones_n)_j)>1$, the update rule of Greenkhorn algorithm implies that 
    \begin{align*}
        h(\bff_{k+1},\bfg_{k+1})-h(\bff_k,\bfg_k)=\gamma\rho(\bfa_I,(\bfP_{k}\ones_n)_I)\geq\frac{\gamma}{n},
    \end{align*}
    which completes the proof.
\end{proof}

\begin{proof}[Proof of lemma \ref{G_decreasing}]
    For any $k\geq 0$, without loss of generality, suppose  the value of $\bfu_I$ is updated at the $k+1$-th iteration. Then we have
    \begin{align}\label{maxfg}
        &\|\bfg^\gamma_*-\bfg_{k+1}\|_\infty = \|\bfg^\gamma_*-\bfg_{k}\|_\infty,\nonumber\\
        &\|\bff^\gamma_*-\bff_{k+1}\|_\infty\leq\max\lb\|\bff^\gamma_*-\bff_{k}\|_\infty,|(\bff_*^\gamma)_I-(\bff_{k+1})_I| \rb.
    \end{align}
    By the update rule  of the Greenkhorn algorithm, we have 
    \begin{align*}
        &\exp\lb\frac{1}{\gamma}(\bff_{k+1})_I\rb = \frac{\bfa_I}{\sum_{j=1}^{n}\bfK_{I,j}(\bfv_{k})_j},\\
        &\exp\lb\frac{1}{\gamma}(\bff_*^\gamma)_I\rb = \frac{\bfa_I}{\sum_{j=1}^{n}\bfK_{I,j}(\bfv_*^\gamma)_j},
    \end{align*}
    where $\bfv_*^\gamma=\exp(\frac{1}{\gamma}\bfg^{\gamma}_*)$ and $\bfv_k=\exp(\frac{1}{\gamma}\bfg_k)$. It follows that
    \begin{align*}
        |(\bff^{\gamma}_*)_I-(\bff_{k+1})_I| = \gamma\left|\log\lb\frac{\sum_{j=1}^{n}\bfK_{I,j}(\bfv_*^\gamma)_j}{\sum_{j=1}^{n}\bfK_{I,j}(\bfv_{k})_j}\rb \right|\leq \gamma\max_j|\log(\bfv^\gamma_*)_j-\log(\bfv_{k})_j| = \|\bfg^\gamma_*-\bfg_{k}\|_\infty.
    \end{align*}
    The proof is complete by combining this result together with \eqref{maxfg}.
\end{proof} 
\end{document}